\documentclass[12pt]{amsart}
\usepackage{amssymb,latexsym,amsmath,amscd,amsthm,graphicx, color}
\usepackage[all]{xy}
\usepackage{pgf,tikz}
\usepackage{mathrsfs}
\usetikzlibrary{arrows}
\definecolor{uuuuuu}{rgb}{0.26666666666666666,0.26666666666666666,0.26666666666666666}
\definecolor{xdxdff}{rgb}{0.49019607843137253,0.49019607843137253,1.}
\definecolor{ffqqqq}{rgb}{1.,0.,0.}

\definecolor{uuuuuu}{rgb}{0.26666666666666666,0.26666666666666666,0.26666666666666666}
\definecolor{qqwuqq}{rgb}{0.,0.39215686274509803,0.}
\definecolor{zzttqq}{rgb}{0.6,0.2,0.}
\definecolor{xdxdff}{rgb}{0.49019607843137253,0.49019607843137253,1.}
\definecolor{qqqqff}{rgb}{0.,0.,1.}
\definecolor{cqcqcq}{rgb}{0.7529411764705882,0.7529411764705882,0.7529411764705882}
\definecolor{sqsqsq}{rgb}{0.12549019607843137,0.12549019607843137,0.12549019607843137}

\theoremstyle{plain}

\newtheorem{theorem}[subsection]{Theorem}

\newtheorem{lemma}[subsection]{Lemma}

\newtheorem{prop}[subsection]{Proposition}
\newtheorem{cor}[subsection]{Corollary}
\newtheorem{example}[subsection]{Example}

\theoremstyle{definition}
\newtheorem{remark}[subsection]{Remark}

\newtheorem{note}[subsection]{Note}

\newcommand{\uu}{\cup}
\newcommand{\ii}{\cap}
\newcommand{\UU}{\bigcup}
\newcommand{\II}{\bigcap}

\newcommand{\sci}{\subset}
\newcommand{\es}{\emptyset}
\newcommand{\set}[1]{\{#1\}}

\newcommand{\ga}{\alpha}
\newcommand{\gb}{\beta}

\newcommand{\gd}{\delta}
\renewcommand{\gg}{\gamma}

\newcommand{\go}{\omega}

\newcommand{\gs}{\sigma}
\newcommand{\gt}{\tau}

\newcommand{\tit}{\textit}

\newcommand{\D}[1]{\mathbb{#1}}
\newcommand{\F}[1]{\mathfrak{#1}}

\newcommand{\te}{\text}

\newcommand{\tri}{\triangle}
\begin{document}

To appear,  Monatshefte f\"ur Mathematik

\title{Quantization for uniform distributions on stretched Sierpi\'nski triangles}

\author{Do\u gan \c C\"omez}
\address{Department of Mathematics \\
408E24 Minard Hall,
North Dakota State University
\\
Fargo, ND 58108-6050, USA.}
\email{Dogan.Comez@ndsu.edu}

\author{ Mrinal Kanti Roychowdhury}
\address{School of Mathematical and Statistical Sciences\\
University of Texas Rio Grande Valley\\
1201 West University Drive\\
Edinburg, TX 78539-2999, USA.}
\email{mrinal.roychowdhury@utrgv.edu}

\subjclass[2010]{60Exx, 28A80, 94A34.}
\keywords{Stretched Sierpi\'nski triangle, probability measure, optimal quantizers, quantization error, quantization dimension, quantization coefficient}
\thanks{The research of the second author was supported by U.S. National Security Agency (NSA) Grant H98230-14-1-0320}

\date{}
\maketitle

\pagestyle{myheadings}\markboth{Do\u gan \c C\"omez and Mrinal Kanti Roychowdhury}{Quantization for uniform distributions on stretched Sierpi\'nski triangles}

\begin{abstract} In this paper, we have considered a uniform probability distribution supported by a stretched Sierpi\'nski triangle. For this probability measure, the optimal sets of $n$-means and the $n$th quantization errors are determined for all $n\geq 2$. In addition, it is shown that the quantization coefficient for such a measure does not exist though the quantization dimension exists.

\end{abstract}

\section{Introduction}  The theory of quantization studies the process of approximating probability measures, which are invariant for certain systems, with discrete probabilities having a finite number of points in their support. Of particular interest are the types of behaviors which may be  encountered in the quantization process for various measures. For an extensive survey of the history of the subject one is referred to \cite{GN}. For mathematical foundation of quantization theory one is referred to \cite{GL2, GL3}. The same mathematical results are used in pattern recognition (optimal sets of prototypes), economics (optimal location of service centers), numerical integration (optimal location of knots) and the theory of convex sets (optimal approximation by polytopes). Let us consider a Borel probability measure $P$ on $\D R^d$ and a natural number $n \in \D N$.  Then, the $n$th \tit{quantization error} for $P$ is defined by:
\[V_n:=V_n(P)=\te{inf}\set{\int \min_{a\in \ga} \|x-a\|^2 dP(x): \ga \sci \D R^d, \, \te{card}(\ga) \leq n},\]
where $\|\cdot\|$ denotes the Euclidean norm on $\D R^d$. A set $\ga$ for which the infimum is achieved and contains no more than $n$ elements is called an \textit{optimal set of $n$-means} for $P$, and the elements in an optimal set are called \tit{optimal quantizers}. Of course, this makes sense only if the mean squared error or the expected squared Euclidean distance $\int \| x\|^2 dP(x)$ is finite (see \cite{AW, GKL, GL1, GL2}). It is known that for a continuous probability measure an optimal set of $n$-means always has exactly $n$-elements (see \cite{GL2}). The number
$D(P):=\lim_{n\to \infty}  \frac{2\log n}{-\log V_n(P)}$, if it exists, is called the \tit{quantization dimension} of the probability measure $P$; on the other hand, for any $s\in (0, +\infty)$, the number $\lim_n n^{\frac 2 s} V_n(P)$, if it exists, is called the $s$-dimensional \tit{quantization coefficient} for $P$. For more details about the quantization dimension and the quantization coefficient, and their connections, one can refer to \cite{GL2, P}.

Let us now state the following proposition (see \cite{GG, GL2}):
\begin{prop} \label{prop10}
Let $\alpha$ be an optimal set of $n$-means and $a\in \ga$. Then,

$(i)$ $P(M(a|\ga))>0$, $(ii)$ $ P(\partial M(a|\ga))=0$, $(iii)$ $a=E(X : X \in M(a|\ga))$,
where $M(a|\ga)$ is the Voronoi region of $a\in \ga , $ i.e.,  $M(a|\ga)$ is the set of all elements $x$ in $\D R^d$ which are closest to $a$ among all the elements in $\ga$.
\end{prop}
\noindent Since
\begin{align*}
E(X : X \in M(a|\ga))=\frac{1}{P(M(a|\ga))}\int_{M(a|\ga)} x dP=\frac{\int_{M(a|\ga)} x dP}{\int_{M(a|\ga)} dP},
\end{align*}
we can say that each element in an optimal set is the centroid, which is actually the conditional expectation, of its own Voronoi region (see also \cite{DFG, R}).

Let $P$ be a Borel probability measure on $\D R$ given by $P=\frac 12 P\circ S_1^{-1}+\frac 12 P\circ S_2^{-1}$ where $S_1(x)=\frac 13x$ and $S_2(x)=\frac 13 x +\frac 23$ for all $x\in \D R$. Then, $P$ has support the classical Cantor set $C$.  For this probability measure Graf and Luschgy gave an exact formula to determine the optimal sets of $n$-means and the $n$th quantization errors for all $n\geq 2$; they also proved that the quantization dimension of this distribution exists and is equal to the Hausdorff dimension $\gb:=\frac{\log 2}{\log 3}$ of the Cantor set, but the $\gb$-dimensional quantization coefficient does not exist \cite{GL3}. In \cite{R1}, the second author gave the bounds of the above exact formula given by Graf and Luschgy. In \cite{R2}, he determined the optimal sets of $n$-means and the $n$th quantization errors for the Cantor distribution generated by infinite similitudes.

Let us now consider a set of three contractive similarity mappings $S_1, S_2, S_3$ on $\D R^2$, such that
 $S_1(x_1, x_2)=\frac 13 (x_1, x_2)$, $S_2(x_1, x_2)=\frac 13 (x_1, x_2)+\frac 23(1, 0)$, and  $S_3(x_1, x_2)=\frac 13 (x_1, x_2)+ \frac 23(\frac 12, \frac {\sqrt{3}}{2})$ for all $(x_1, x_2) \in \D R^2$.  The limit set $S$ of the iterated function system $\{S_i \}_{i=1}^3 $ is a version of the Sierpi\`nski triangle, which is constructed as follows:
$(i)$ Start with an equilateral triangle;  $(ii)$ delete the open middle third from each side of the triangle and join the end points of the adjacent sides to construct three smaller congruent equilateral triangles;  $(iii)$ repeat step (ii) with each of the remaining smaller triangles.  At each step the new triangles appear as radiated from the center of the triangle in the previous step towards the vertices. In order to distinguish it from the classical Sierpi\`nski triangle we will call it the \tit{stretched Sierpi\`nski triangle}.
It is easy to see that the area and the circumference of a stretched Sierpi\`nski triangle are zero and it has the Hausdorff dimension one (see also Section 4).  Let $P=\frac 13\mathop{\sum}_{j=1}^3 P\circ S_j^{-1}$. Then, $P$ is a unique Borel probability measure on $\D R^2$ with support the stretched Sierpi\`nski triangle generated by $S_1, S_2, S_3$.  For this probability measure $P$, in this paper, we determine the optimal sets of $n$-means and the $n$th quantization errors for all $n\geq 2$. In Theorem~\ref{Th4}, we further show that although the quantization dimension exists, the quantization coefficient for the probability measure $P$ does not exist.

\section{Basic definitions and lemmas}
In this section, we give the basic definitions and lemmas that will be instrumental in our analysis. Let $S_1, S_2$ and $S_3$ be the generating maps of the stretched Sierpi\`nski triangle as defined in the previous section. By a word $\go$ of length $k$ over the alphabet $I:=\set{1, 2, 3}$, it is meant that $\go:=\go_1\go_2\cdots \go_k \in I^k$, where $k\geq 1$. A word of length zero is called the empty word and is denoted by $\es$. By $I^\ast$, we denote the set of all words over the alphabet $I$ including the empty word $\es$. By the concatenation of two words $\go:=\go_1\go_2\cdots \go_k$ and $\gt:=\gt_1\gt_2\cdots \gt_\ell$, denoted by $\go\gt$, it is meant $\go\gt:=\go_1\cdots \go_k\gt_1\cdots \gt_\ell$. For $\go=\go_1\go_2 \cdots\go_k \in I^k$, set $S_\go:=S_{\go_1} \circ \cdots \circ S_{\go_k}$. Let $\tri$ be the equilateral triangle with vertices $(0, 0)$, $(1, 0)$ and $(\frac 12, \frac {\sqrt{3}} {2})$. The sets $\set{\tri_\go : \go \in I^k}$ are just the $3^k$ triangles in the $k$th level in the construction of the stretched Sierpi\`nski triangle. The triangles $\tri_{\go1}$, $\tri_{\go2}$ and $\tri_{\go3}$ into which $\tri_\go$ is split up at the $(k+1)$th level are called the \tit{basic triangles} of $\tri_\go$. The set $S:=\II_{k \in \D N} \UU_{\go \in I^k} \tri_\go$ is the stretched Sierpi\`nski triangle and equals the support of the probability measure $P$ given by $P=\frac 13\mathop{\sum}\limits_{j=1}^3 P\circ S_j^{-1}$. It follows that, by induction,
$P=\frac 1 {3^k} \sum_{\go \in I^k}  P\circ S_\go^{-1} $ for any $\go\in I^k$, where $k\geq 1$.  Hence, we have the following lemma:

\begin{lemma} \label{lemma1} Let $f: \D R \to \D R^+$ be Borel measurable and $k\in \D N$. Then, for all $\go \in I^*, $
\[\int f \,dP=\frac 1 {3^k} \sum_{\go \in I^k}\int f\circ S_\go \,dP.\]
\end{lemma}
Let $S_{(i1)}, \, S_{(i2)}$ be the horizontal and vertical components of the transformations $S_i$ for $1\leq i\leq 3$. Then, for any $(x_1, x_2) \in \D R^2$ we have
$S_{(11)}(x_1) =\frac 1 3 x_1$, $ S_{(12)}(x_2)=\frac 1 3 x_2$, $S_{(21)}(x_1)=\frac 1 3 x_1 +\frac 23$, $S_{(22)}(x_2)=\frac 1 3 x_2$, $S_{(31)}(x_1)=\frac 1 3 x_1 +\frac 13$, and $S_{(32)}(x_2)= \frac 1 3 x_2+ \frac {\sqrt{3}} 3$. Let $X:=(X_1, X_2)$ be a bivariate random variable with distribution $P$. Let $P_1, P_2$ be the marginal distributions of $P$, i.e., $P_1(A)=P(A\times \D R)=P\circ \pi_1^{-1} (A) $ for all $A \in \F B$, and $P_2(B)=P(\D R \times B)=P\circ \pi_2^{-1} (B)$ for all $B \in \F B$, where $\pi_1, \pi_2$ are two projection mappings given by $\pi_1(x_1, x_2)=x_1$ and $\pi_2 (x_1, x_2)=x_2$ for all $(x_1, x_2) \in \D R^2$, and $\F B$ is the Borel $\gs$-algebra on $\D R$. Then $X_1$ has distribution $P_1$ and $X_2$ has distribution $P_2$.

The statement below provides the connection between $P$ and its marginal distributions via the components of the generating maps $S_i$. The proof is not difficult to see.

\begin{lemma} Let $P_1$ and $P_2$ be the marginal distributions of the probability measure $P$. Then,
\begin{itemize}
\item[•]
$P_1 =\frac 1 3 P_1 \circ S_{(11)}^{-1} + \frac 1 3 P_1\circ S_{(21)}^{-1} +\frac 1 3  P_1\circ S_{(31)}^{-1}$  and
$P_2 =\frac 1 3 P_2 \circ S_{(12)}^{-1} + \frac 1 3P_2\circ S_{(22)}^{-1} +\frac 1 3  P_2\circ S_{(32)}^{-1}$.
\end{itemize}
\end{lemma}

For words $\gb, \gg, \cdots, \gd$ in $I^\ast$, by $a(\gb, \gg, \cdots, \gd)$ we mean the conditional expectation of the random variable $X$ given $\tri_\gb\uu \tri_\gg \uu\cdots \uu \tri_\gd,$ i.e.,
\begin{equation} \label{eq0}a(\gb, \gg, \cdots, \gd)=E(X|X\in \tri_\gb \uu \tri_\gg \uu \cdots \uu \tri_\gd)=\frac{1}{P(\tri_\gb\uu \cdots \uu \tri_\gd)}\int_{\tri_\gb\uu \cdots \uu \tri_\gd} x dP.
\end{equation}
\begin{lemma} \label{lemma333} Let $E(X)$ and $V(X)$ denote the the expectation and the variance of the random variable $X$. Then, \[E(X)=(E(X_1), \, E(X_2))=(\frac 12, \frac {\sqrt{3}} 6) \te{ and } V:=V(X)=E\|X-(\frac 1 2, \frac 1 2)\|^2=\frac 1 6.\]
\end{lemma}

\begin{proof} We have
\begin{align*}
E(X_1)&=\int x_1 dP_1
=\frac 1 3 \Big[ \int x_1 dP_1 \circ S_{(11)}^{-1} +  \int x_1 dP_1\circ S_{(21)}^{-1}+  \int x_1 dP_1\circ S_{(31)}^{-1} \Big] \\
&=\frac 1 3 \Big[ \int \frac 13 x_1 dP_1 +  \int (\frac 1 3 x_1+\frac  23) dP_1 + \int (\frac 13 x_1+\frac 13) dP_1 \Big],
\end{align*}
which implies $E(X_1)=\frac 1 2$ and similarly, one can show that $E(X_2)=\frac {\sqrt 3}{6}$. Now,
\begin{align*}
E(X_1^2)&=\int x_1^2 \, dP_1=\frac 1 3 \int x_1^2\, dP_1 \circ S_{(11)}^{-1} + \frac 1 3 \int x_1^2 dP_1\circ S_{(21)}^{-1} +\frac 1 3  \int x_1^2 dP_1\circ S_{(31)}^{-1}\\
&=\frac 1 3 \Big[ \int (\frac 13 x_1)^2 dP_1 +  \int (\frac 1 3 x_1+\frac  23)^2dP_1 +  \int (\frac 13 x_1+\frac 1 3)^2  dP_1 \Big] \\
&=\frac 3{27} E(X_1^2)+\frac 6{27} E(X_1) +\frac {5}{27}=\frac 1 9 E(X_1^2)+\frac 8{27},
\end{align*}
which implies  $E(X_1^2)=\frac 1 {3}$. Thus, we see that
$V(X_1)=E(X_1^2)-(E(X_1))^2=\frac{1}{3} - \frac 1 4 =\frac 1 {12}$. Similarly, one can show that $V(X_2)=\frac 1{12}$. Hence,
 \begin{align*} & E\|X-(\frac 12, \frac {\sqrt{3}}{6})\|^2=\iint_{\D R^2} \Big((x_1-\frac 12)^2+(x_2-  \frac {\sqrt{3}}{6})^2\Big) dP(x_1, x_2)\\
&=\int(x_1-\frac 12)^2 dP_1(x_1)+\int (x_2- \frac {\sqrt{3}}{6})^2 dP_2(x_2)=V(X_1)+V(X_2)=\frac 1 6,
\end{align*}
which completes the proof of the lemma.
\end{proof}

\begin{note} \label{note1} From Lemma~\ref{lemma333}  it follows that the optimal set of one-mean is the expected value and the corresponding quantization error is the variance $V$ of the random variable $X$.
For $\go \in I^k$, $k\geq 1$, since $a(\go)=E(X : X \in J_\go)$, using Lemma~\ref{lemma1}, we have
\begin{align*}
&a(\go)=\frac{1}{P(\tri_\go)} \int_{\tri_\go} x \,dP(x)=\int_{\tri_\go} x\, dP\circ S_\go^{-1}(x)=\int S_\go(x)\, dP(x)=E(S_\go(X))=S_\go(\frac 12, \frac {\sqrt{3}}{6}).
\end{align*}  For any $(a, b) \in \D R^2$, $E\|X-(a, b)\|^2=V+\|(\frac 12, \frac {\sqrt{3}} 6)-(a, b)\|^2.$
In fact, for any $\go \in I^k$, $k\geq 1$, we have
$\int_{\tri_\go}\|x-(a, b)\|^2 dP= \frac{1}{3^k} \int\|(x_1, x_2) -(a, b)\|^2 dP\circ S_\go^{-1},$
which implies
\begin{equation} \label{eq1}
\int_{\tri_\go}\|x-(a, b)\|^2 dP=\frac{1}{3^k}\Big(\frac{1}{9^k}V+\|a(\go)-(a, b)\|^2\Big).
\end{equation}
\end{note}

\bigskip

\section{Optimal sets of $n$-means for all $n\geq 2$}
 Recall that $\ga_n$ represents an optimal set of $n$-means for all $n\geq 1$. For $k\geq 0$ and $\go \in I^k$,  by $a(\go)$ it is meant $a(\go)=S_\go(E(X))$. Also, recall the notation given by \eqref{eq0}. The work in this section involves some straightforward and lengthy computations, for which, in some parts, we have used Mathematica. For the readers' convenience, in Section~\ref{sec4}, we have given the Mathematica code and user guide to let the readers know how the code was used in computations.

Below, when we state that the stretched Sierpi\`nski triangle is  symmetric with respect to the probability distribution $P$, it is meant that if the two basic triangles of similar geometrical shape lie in the opposite sides of a median, and are equidistant from the median, then they have the same probability.

\begin{prop}\label{prop1} The set $\set{(\frac{1}{2},\frac{7}{6 \sqrt{3}}), (\frac{1}{2},\frac{1}{6 \sqrt{3}})}$ is an optimal set of two-means with quantization error $V_2=\frac{5}{54}=0.0925926$.
\end{prop}

\begin{proof}
The stretched Sierpi\`nski triangle with respect to any of its medians has the maximum symmetry, i.e., with respect to any of its medians the stretched Sierpi\`nski triangle is geometrically symmetric as well as symmetric with respect to the probability distribution $P$.  Due to this fact, among all the pairs of two points which have the boundaries of the Voronoi regions oblique lines passing through the centroid $(\frac{1}{2},\frac{\sqrt{3}}{6})$, the two points which have the boundary of the Voronoi regions the line perpendicular to a median will give the smallest distortion error. Without any loss of generality, to get an optimal set of two-means we consider the median passing through the vertex $(\frac 12, \frac {\sqrt{3}}2)$. Let $\ga:=\{(p, b_1), (p, b_2)\}$ be an optimal set of two-means with $b_1\leq b_2$. Since the optimal quantizers are the centroids of their own Voronoi regions, by the properties of centroids, we have
\[(p, b_1) P(M((p, b_1)|\ga))+(p, b_2) P(M((p, b_2)|\ga))=(\frac 12, \frac {\sqrt{3}} 6),\]
which implies $p=\frac 12$ and $b_1 P(M((p, b_1)|\ga))+b_2 P(M((p, b_2)|\ga))=\frac {\sqrt{3}} 6$. Thus, it follows that the two optimal quantizers are $(\frac 12, b_1)$ and $(\frac 12, b_2)$, and they lie in the opposite sides of the point $(\frac 12, \frac {\sqrt{3}} 6)$. This yields the fact that $\tri_1\uu \tri_2\sci M((\frac 12, b_1)|\ga)$ and $\tri_3 \sci M((\frac 12, b_2)|\ga)$. Again, the optimal quantizers are the centroids of their own Voronoi regions, and so by equation~\eqref{eq0}, we have
\begin{align*}
&(\frac 12, b_1)=E(X : X\in \tri_1\uu \tri_2)= \frac 1{\frac 1 3+\frac 13}\Big(\frac 13 a(1)+\frac 13 a(2)\Big)=(\frac{1}{2},\frac{1}{6 \sqrt{3}}),\\
&(\frac 12, b_2)=E(X : X\in  \tri_3)=a(3)=(\frac{1}{2},\frac{7}{6 \sqrt{3}}),
\end{align*}
and then the quantization error is
\begin{align*}
&V_2=\mathop{\int}\limits_{\tri_1\uu \tri_2}\min_{a\in \ga}\|x-a\|^2 dP+\mathop{\int}\limits_{\tri_3}\min_{a\in \ga}\|x-a\|^2 dP\\
&=\mathop{\int}\limits_{\tri_1}\|x-(\frac{1}{2},\frac{1}{6 \sqrt{3}})\|^2 dP+\mathop{\int}\limits_{\tri_2}\|x-(\frac{1}{2},\frac{1}{6 \sqrt{3}})\|^2 dP+\mathop{\int}\limits_{\tri_3}\|x-(\frac{1}{2},\frac{7}{6 \sqrt{3}})\|^2 dP\\
&=\frac{5}{54}=0.0925926.
\end{align*}
Hence, the proof of the proposition is complete.
\end{proof}

\begin{remark} Due to symmetry, the sets $\set{(\frac{1}{6},\frac{1}{6 \sqrt{3}}), (\frac{2}{3},\frac{2}{3 \sqrt{3}})}$, and $\set{(\frac{5}{6}, \frac{1}{6 \sqrt{3}}), (\frac{1}{3},  \frac{2}{3 \sqrt{3}})}$ also form optimal sets of two-means with quantization error $V_2=\frac 5{54}$  (see Figure~\ref{Fig}).
\end{remark}

\begin{figure}
\begin{tikzpicture}[line cap=round,line join=round,>=triangle 45,x=0.3 cm,y=0.3 cm]
\clip(-0.717946985327632,-0.81755092664275) rectangle (12.8592934709513,11.586582306908248);
\draw (0.,0.)-- (6.,10.392304845413264);
\draw (6.,10.392304845413264)-- (12.,0.);
\draw (0.,0.)-- (12.,0.);
\draw (2.,3.4641016151377553)-- (4.,0.);
\draw (10.,3.4641016151377553)-- (8.,0.);
\draw (4.,6.9282032302755105)-- (8.,6.9282032302755105);
\draw (0.6666666666666666,1.1547005383792515)-- (1.3333333333333333,0.);
\draw (3.3333333333333335,1.1547005383792515)-- (2.6666666666666665,0.);
\draw (1.3333333333333333,2.309401076758503)-- (2.6666666666666665,2.309401076758503);
\draw (8.666666666666666,1.1547005383792515)-- (9.333333333333334,0.);
\draw (11.333333333333332,1.1547005383792515)-- (10.666666666666666,0.);
\draw (9.333333333333334,2.309401076758503)-- (10.666666666666666,2.309401076758503);
\draw (4.666666666666667,8.082903768654761)-- (5.333333333333333,6.9282032302755105);
\draw (7.333333333333334,8.082903768654761)-- (6.666666666666667,6.9282032302755105);
\draw (5.333333333333333,9.237604307034012)-- (6.666666666666667,9.237604307034012);
\begin{scriptsize}
\draw [fill=ffqqqq] (6.,3.4641016151377544) circle (2.5pt);
\end{scriptsize}
\end{tikzpicture}
 \begin{tikzpicture}[line cap=round,line join=round,>=triangle 45,x=0.3 cm,y=0.3 cm]
\clip(-0.717946985327634,-0.8175509266427525) rectangle (12.859293470951293,11.58658230690825);
\draw (0.,0.)-- (6.,10.392304845413264);
\draw (6.,10.392304845413264)-- (12.,0.);
\draw (0.,0.)-- (12.,0.);
\draw (2.,3.4641016151377553)-- (4.,0.);
\draw (10.,3.4641016151377553)-- (8.,0.);
\draw (4.,6.9282032302755105)-- (8.,6.9282032302755105);
\draw (0.6666666666666666,1.1547005383792515)-- (1.3333333333333333,0.);
\draw (3.3333333333333335,1.1547005383792515)-- (2.6666666666666665,0.);
\draw (1.3333333333333333,2.309401076758503)-- (2.6666666666666665,2.309401076758503);
\draw (8.666666666666666,1.1547005383792515)-- (9.333333333333334,0.);
\draw (11.333333333333332,1.1547005383792515)-- (10.666666666666666,0.);
\draw (9.333333333333334,2.309401076758503)-- (10.666666666666666,2.309401076758503);
\draw (4.666666666666667,8.082903768654761)-- (5.333333333333333,6.9282032302755105);
\draw (7.333333333333334,8.082903768654761)-- (6.666666666666667,6.9282032302755105);
\draw (5.333333333333333,9.237604307034012)-- (6.666666666666667,9.237604307034012);
\begin{scriptsize}
\draw [fill=ffqqqq] (6.,1.1547005383792515) circle (2.5pt);
\draw [fill=ffqqqq] (6.,8.082903768654761) circle (2.5pt);
\end{scriptsize}
\end{tikzpicture}
\begin{tikzpicture}[line cap=round,line join=round,>=triangle 45,x=0.3 cm,y=0.3 cm]
\clip(-0.717946985327634,-0.8175509266427525) rectangle (12.859293470951293,11.58658230690825);
\draw (0.,0.)-- (6.,10.392304845413264);
\draw (6.,10.392304845413264)-- (12.,0.);
\draw (0.,0.)-- (12.,0.);
\draw (2.,3.4641016151377553)-- (4.,0.);
\draw (10.,3.4641016151377553)-- (8.,0.);
\draw (4.,6.9282032302755105)-- (8.,6.9282032302755105);
\draw (0.6666666666666666,1.1547005383792515)-- (1.3333333333333333,0.);
\draw (3.3333333333333335,1.1547005383792515)-- (2.6666666666666665,0.);
\draw (1.3333333333333333,2.309401076758503)-- (2.6666666666666665,2.309401076758503);
\draw (8.666666666666666,1.1547005383792515)-- (9.333333333333334,0.);
\draw (11.333333333333332,1.1547005383792515)-- (10.666666666666666,0.);
\draw (9.333333333333334,2.309401076758503)-- (10.666666666666666,2.309401076758503);
\draw (4.666666666666667,8.082903768654761)-- (5.333333333333333,6.9282032302755105);
\draw (7.333333333333334,8.082903768654761)-- (6.666666666666667,6.9282032302755105);
\draw (5.333333333333333,9.237604307034012)-- (6.666666666666667,9.237604307034012);
\begin{scriptsize}
\draw [fill=ffqqqq] (6.,8.082903768654761) circle (2.5pt);
\draw [fill=ffqqqq] (2.,1.1547005383792515) circle (2.5pt);
\draw [fill=ffqqqq] (10.,1.1547005383792515) circle (2.5pt);
\end{scriptsize}
\end{tikzpicture}

\vspace{0.1 in}
\begin{tikzpicture}[line cap=round,line join=round,>=triangle 45, x=0.3 cm,y=0.3 cm]
\clip(-0.717946985327634,-0.8175509266427525) rectangle (12.859293470951293,11.58658230690825);
\draw (0.,0.)-- (6.,10.392304845413264);
\draw (6.,10.392304845413264)-- (12.,0.);
\draw (0.,0.)-- (12.,0.);
\draw (2.,3.4641016151377553)-- (4.,0.);
\draw (10.,3.4641016151377553)-- (8.,0.);
\draw (4.,6.9282032302755105)-- (8.,6.9282032302755105);
\draw (0.6666666666666666,1.1547005383792515)-- (1.3333333333333333,0.);
\draw (3.3333333333333335,1.1547005383792515)-- (2.6666666666666665,0.);
\draw (1.3333333333333333,2.309401076758503)-- (2.6666666666666665,2.309401076758503);
\draw (8.666666666666666,1.1547005383792515)-- (9.333333333333334,0.);
\draw (11.333333333333332,1.1547005383792515)-- (10.666666666666666,0.);
\draw (9.333333333333334,2.309401076758503)-- (10.666666666666666,2.309401076758503);
\draw (4.666666666666667,8.082903768654761)-- (5.333333333333333,6.9282032302755105);
\draw (7.333333333333334,8.082903768654761)-- (6.666666666666667,6.9282032302755105);
\draw (5.333333333333333,9.237604307034012)-- (6.666666666666667,9.237604307034012);
\begin{scriptsize}
\draw [fill=ffqqqq] (6.,8.082903768654761) circle (2.5pt);
\draw [fill=ffqqqq] (2.000000000, 2.694301256) circle (2.5pt);
\draw [fill=ffqqqq] (2.000000000, 0.3849001795) circle (2.5pt);
\draw [fill=ffqqqq] (10.,1.1547005383792515) circle (2.5pt);
\end{scriptsize}
\end{tikzpicture}
\begin{tikzpicture}[line cap=round,line join=round,>=triangle 45,x=0.3 cm,y=0.3 cm]
\clip(-0.717946985327634,-0.8175509266427525) rectangle (12.859293470951293,11.58658230690825);
\draw (0.,0.)-- (6.,10.392304845413264);
\draw (6.,10.392304845413264)-- (12.,0.);
\draw (0.,0.)-- (12.,0.);
\draw (2.,3.4641016151377553)-- (4.,0.);
\draw (10.,3.4641016151377553)-- (8.,0.);
\draw (4.,6.9282032302755105)-- (8.,6.9282032302755105);
\draw (0.6666666666666666,1.1547005383792515)-- (1.3333333333333333,0.);
\draw (3.3333333333333335,1.1547005383792515)-- (2.6666666666666665,0.);
\draw (1.3333333333333333,2.309401076758503)-- (2.6666666666666665,2.309401076758503);
\draw (8.666666666666666,1.1547005383792515)-- (9.333333333333334,0.);
\draw (11.333333333333332,1.1547005383792515)-- (10.666666666666666,0.);
\draw (9.333333333333334,2.309401076758503)-- (10.666666666666666,2.309401076758503);
\draw (4.666666666666667,8.082903768654761)-- (5.333333333333333,6.9282032302755105);
\draw (7.333333333333334,8.082903768654761)-- (6.666666666666667,6.9282032302755105);
\draw (5.333333333333333,9.237604307034012)-- (6.666666666666667,9.237604307034012);
\begin{scriptsize}
\draw [fill=ffqqqq] (6.,8.082903768654761) circle (2.5pt);
\draw [fill=ffqqqq] (2.000000000, 2.694301256) circle (2.5pt);
\draw [fill=ffqqqq] (2.000000000, 0.3849001795) circle (2.5pt);
\draw [fill=ffqqqq] (10.00000000, 2.694301256) circle (2.5pt);
\draw [fill=ffqqqq] (10.00000000, 0.3849001795) circle (2.5pt);
\end{scriptsize}
\end{tikzpicture}  \begin{tikzpicture}[line cap=round,line join=round,>=triangle 45,x=0.3 cm,y=0.3 cm]
\clip(-0.717946985327634,-0.8175509266427525) rectangle (12.859293470951293,11.58658230690825);
\draw (0.,0.)-- (6.,10.392304845413264);
\draw (6.,10.392304845413264)-- (12.,0.);
\draw (0.,0.)-- (12.,0.);
\draw (2.,3.4641016151377553)-- (4.,0.);
\draw (10.,3.4641016151377553)-- (8.,0.);
\draw (4.,6.9282032302755105)-- (8.,6.9282032302755105);
\draw (0.6666666666666666,1.1547005383792515)-- (1.3333333333333333,0.);
\draw (3.3333333333333335,1.1547005383792515)-- (2.6666666666666665,0.);
\draw (1.3333333333333333,2.309401076758503)-- (2.6666666666666665,2.309401076758503);
\draw (8.666666666666666,1.1547005383792515)-- (9.333333333333334,0.);
\draw (11.333333333333332,1.1547005383792515)-- (10.666666666666666,0.);
\draw (9.333333333333334,2.309401076758503)-- (10.666666666666666,2.309401076758503);
\draw (4.666666666666667,8.082903768654761)-- (5.333333333333333,6.9282032302755105);
\draw (7.333333333333334,8.082903768654761)-- (6.666666666666667,6.9282032302755105);
\draw (5.333333333333333,9.237604307034012)-- (6.666666666666667,9.237604307034012);
\begin{scriptsize}
\draw [fill=ffqqqq] (6.000000000, 9.622504486) circle (2.5pt);
\draw [fill=ffqqqq] (6.000000000, 7.313103410) circle (2.5pt);
\draw [fill=ffqqqq] (2.000000000, 2.694301256) circle (2.5pt);
\draw [fill=ffqqqq] (2.000000000, 0.3849001795) circle (2.5pt);
\draw [fill=ffqqqq] (10.00000000, 2.694301256) circle (2.5pt);
\draw [fill=ffqqqq] (10.00000000, 0.3849001795) circle (2.5pt);
\end{scriptsize}
\end{tikzpicture}

\vspace{0.1 in}
\begin{tikzpicture}[line cap=round,line join=round,>=triangle 45, x=0.3 cm,y=0.3 cm]
\clip(-0.717946985327634,-0.8175509266427525) rectangle (12.859293470951293,11.58658230690825);
\draw (0.,0.)-- (6.,10.392304845413264);
\draw (6.,10.392304845413264)-- (12.,0.);
\draw (0.,0.)-- (12.,0.);
\draw (2.,3.4641016151377553)-- (4.,0.);
\draw (10.,3.4641016151377553)-- (8.,0.);
\draw (4.,6.9282032302755105)-- (8.,6.9282032302755105);
\draw (0.6666666666666666,1.1547005383792515)-- (1.3333333333333333,0.);
\draw (3.3333333333333335,1.1547005383792515)-- (2.6666666666666665,0.);
\draw (1.3333333333333333,2.309401076758503)-- (2.6666666666666665,2.309401076758503);
\draw (8.666666666666666,1.1547005383792515)-- (9.333333333333334,0.);
\draw (11.333333333333332,1.1547005383792515)-- (10.666666666666666,0.);
\draw (9.333333333333334,2.309401076758503)-- (10.666666666666666,2.309401076758503);
\draw (4.666666666666667,8.082903768654761)-- (5.333333333333333,6.9282032302755105);
\draw (7.333333333333334,8.082903768654761)-- (6.666666666666667,6.9282032302755105);
\draw (5.333333333333333,9.237604307034012)-- (6.666666666666667,9.237604307034012);
\begin{scriptsize}
\draw [fill=ffqqqq] (6.000000000, 9.622504486) circle (2.5pt);
\draw [fill=ffqqqq] (6.000000000, 7.313103410) circle (2.5pt);
\draw [fill=ffqqqq] (2.000000000, 2.694301256) circle (2.5pt);
\draw [fill=ffqqqq] (0.6666666667, 0.3849001795) circle (2.5pt);
\draw [fill=ffqqqq] (3.333333333, 0.3849001795) circle (2.5pt);
\draw [fill=ffqqqq] (10.00000000, 2.694301256) circle (2.5pt);
\draw [fill=ffqqqq] (10.00000000, 0.3849001795) circle (2.5pt);
\end{scriptsize}
\end{tikzpicture}
\begin{tikzpicture}[line cap=round,line join=round,>=triangle 45,x=0.3 cm,y=0.3 cm]
\clip(-0.717946985327634,-0.8175509266427525) rectangle (12.859293470951293,11.58658230690825);
\draw (0.,0.)-- (6.,10.392304845413264);
\draw (6.,10.392304845413264)-- (12.,0.);
\draw (0.,0.)-- (12.,0.);
\draw (2.,3.4641016151377553)-- (4.,0.);
\draw (10.,3.4641016151377553)-- (8.,0.);
\draw (4.,6.9282032302755105)-- (8.,6.9282032302755105);
\draw (0.6666666666666666,1.1547005383792515)-- (1.3333333333333333,0.);
\draw (3.3333333333333335,1.1547005383792515)-- (2.6666666666666665,0.);
\draw (1.3333333333333333,2.309401076758503)-- (2.6666666666666665,2.309401076758503);
\draw (8.666666666666666,1.1547005383792515)-- (9.333333333333334,0.);
\draw (11.333333333333332,1.1547005383792515)-- (10.666666666666666,0.);
\draw (9.333333333333334,2.309401076758503)-- (10.666666666666666,2.309401076758503);
\draw (4.666666666666667,8.082903768654761)-- (5.333333333333333,6.9282032302755105);
\draw (7.333333333333334,8.082903768654761)-- (6.666666666666667,6.9282032302755105);
\draw (5.333333333333333,9.237604307034012)-- (6.666666666666667,9.237604307034012);
\begin{scriptsize}
\draw [fill=ffqqqq] (6.000000000, 9.622504486) circle (2.5pt);
\draw [fill=ffqqqq] (6.000000000, 7.313103410) circle (2.5pt);
\draw [fill=ffqqqq] (2.000000000, 2.694301256) circle (2.5pt);
\draw [fill=ffqqqq] (0.6666666667, 0.3849001795) circle (2.5pt);
\draw [fill=ffqqqq] (3.333333333, 0.3849001795) circle (2.5pt);
\draw [fill=ffqqqq] (10.00000000, 2.694301256) circle (2.5pt);
\draw [fill=ffqqqq] (8.666666667, 0.3849001795) circle (2.5pt);
\draw [fill=ffqqqq] (11.33333333, 0.3849001795) circle (2.5pt);
\end{scriptsize}
\end{tikzpicture}  \begin{tikzpicture}[line cap=round,line join=round,>=triangle 45,x=0.3 cm,y=0.3 cm]
\clip(-0.717946985327634,-0.8175509266427525) rectangle (12.859293470951293,11.58658230690825);
\draw (0.,0.)-- (6.,10.392304845413264);
\draw (6.,10.392304845413264)-- (12.,0.);
\draw (0.,0.)-- (12.,0.);
\draw (2.,3.4641016151377553)-- (4.,0.);
\draw (10.,3.4641016151377553)-- (8.,0.);
\draw (4.,6.9282032302755105)-- (8.,6.9282032302755105);
\draw (0.6666666666666666,1.1547005383792515)-- (1.3333333333333333,0.);
\draw (3.3333333333333335,1.1547005383792515)-- (2.6666666666666665,0.);
\draw (1.3333333333333333,2.309401076758503)-- (2.6666666666666665,2.309401076758503);
\draw (8.666666666666666,1.1547005383792515)-- (9.333333333333334,0.);
\draw (11.333333333333332,1.1547005383792515)-- (10.666666666666666,0.);
\draw (9.333333333333334,2.309401076758503)-- (10.666666666666666,2.309401076758503);
\draw (4.666666666666667,8.082903768654761)-- (5.333333333333333,6.9282032302755105);
\draw (7.333333333333334,8.082903768654761)-- (6.666666666666667,6.9282032302755105);
\draw (5.333333333333333,9.237604307034012)-- (6.666666666666667,9.237604307034012);
\begin{scriptsize}
\draw [fill=ffqqqq] (6.000000000, 9.622504486) circle (2.5pt);
\draw [fill=ffqqqq] (4.666666667, 7.313103410) circle (2.5pt);
\draw [fill=ffqqqq] (7.333333333, 7.313103410) circle (2.5pt);
\draw [fill=ffqqqq] (2.000000000, 2.694301256) circle (2.5pt);
\draw [fill=ffqqqq] (0.6666666667, 0.3849001795) circle (2.5pt);
\draw [fill=ffqqqq] (3.333333333, 0.3849001795) circle (2.5pt);
\draw [fill=ffqqqq] (10.00000000, 2.694301256) circle (2.5pt);
\draw [fill=ffqqqq] (8.666666667, 0.3849001795) circle (2.5pt);
\draw [fill=ffqqqq] (11.33333333, 0.3849001795) circle (2.5pt);
\end{scriptsize}
\end{tikzpicture}
\caption{Optimal configuration of $n$ points for $1\leq n\leq 9$ on the stretched Sierpi\`nski triangle.} \label{Fig}
\end{figure}
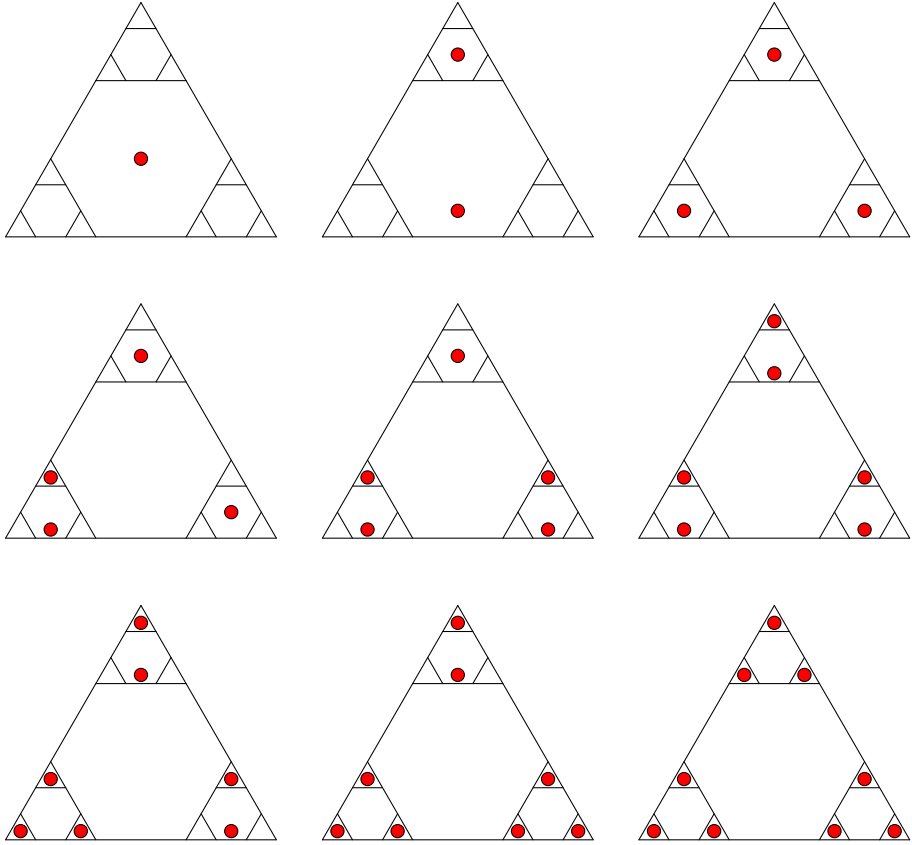

\begin{lemma} \label{lemma30}
Let $\ga$ be an optimal set of $n$-means with $n\geq 3$. Then, $\ga\ii \tri_i\neq \es$ for all $1\leq i\leq 3$.
\end{lemma}

\begin{proof} Let us consider an arbitrary three-point set $\gb$ given by $\gb=\set{a(1), a(2), a(3)}$. Then, the distortion error is
\begin{align*}
\int\min_{a\in \ga} \|x-a\|^2 dP=\mathop{\sum}_{i=1}^3 \mathop{\int}\limits_{\tri_i}\|x-a(i)\|^2 dP=3 \cdot\frac{1}{162}=\frac 1{54}=0.0185185.
\end{align*}
Since, $V_n$ is the quantization error for $n\geq 3$, we have $0.0185185\geq V_3\geq V_n$. Let $\ga$ be an optimal set of $n$-means for $n\geq 3$.
As the optimal quantizers are the centroids of their own Voronoi regions, we have $\ga \sci \tri$.

Suppose that $\ga$ does not contain any point from $\mathop{\uu}\limits_{i=1}^3\tri_i$.
If all the points of $\ga$ are below the line $x_2=\frac {2\sqrt{3}}{9}$, for any $(x_1, x_2) \in \tri_3$ we have $\min_{(a, b)\in \ga}\|(x_1, x_2)-(a, b)\|^2\geq (\frac{\sqrt{3}}{3}-\frac{2\sqrt{3}}{9})^2=\frac1{27},$ and for any $(x_1, x_2) \in \tri_{11}\uu\tri_{22}$ we have $\min_{(a, b)\in \ga}\|(x_1, x_2)-(a, b)\|^2\geq \frac1{27}$, and then the distortion error is obtained as
\begin{align*}
&\int\min_{a\in \ga}\|x-a\|^2 dP>\mathop{\int}\limits_{\tri_3}\min_{a\in \ga}\|x-a\|^2 dP +\mathop{\int}\limits_{\tri_{11}\uu\tri_{22}} \min_{a\in \ga}\|x-a\|^2 dP\\
&\geq \frac 13 \cdot\frac{1}{27}+2 \cdot\frac 19 \cdot\frac{1}{27} =\frac{5}{243}=0.0205761>V_3,
\end{align*}
which is a contradiction. If $\ga$ does not contain any point below the line  $x_2=\frac {2\sqrt{3}}{9}$, for any $(x_1, x_2) \in \tri_{11}\uu\tri_{12}\uu\tri_{21}\uu\tri_{22}$ we have $\min_{a\in \ga}\|(x_1, x_2)-a\|^2\geq \frac1{12}$, and then the distortion error is obtained as
\begin{align*}
&\int\min_{a\in \ga}\|x-a\|^2 dP>\mathop{\int}\limits_{\mathop{\uu}\limits_{i, j=1}^2\tri_{ij}}\min_{a\in \ga}\|x-a\|^2 dP
\geq 4\cdot \frac 19\cdot \frac 1{12} =\frac{1}{27}=0.037037>V_3,
\end{align*}
 which is a contradiction as well. Thus, we conclude that $\ga$ contains points both above and below the line  $x_2=\frac {2\sqrt{3}}{9}$. If $\ga$ contains two or more points below the line $x_2=\frac {2\sqrt{3}}{9}$, then the quantization error can be strictly reduced by moving points below the line $x_2=\frac {2\sqrt{3}}{9}$ to $\tri_1$ and $\tri_2$, and by moving the points above the line $x_2=\frac {2\sqrt{3}}{9}$ to $\tri_3$, and so, we assume that $\ga$ contains only one point below the line $x_2=\frac {2\sqrt{3}}{9}$. Due to symmetry we can assume that this point lies on the line $x_1=\frac 12$.
Then, notice that $a(12, 21)=(\frac{1}{2},\frac{1}{18 \sqrt{3}})$ and it is the midpoint of the line segment joining the centroids of $\tri_{12}$ and $\tri_{21}$; the point of intersection of the lines $x_2=\sqrt{3} x_1$ and $x_2=\frac{2 \sqrt{3}}{9}$ is $(\frac{2}{9}, \frac{2 \sqrt{3}}{9})$, and the base of the perpendicular passing through $(0, 0)$ of the triangle $\tri_1$ is $(\frac{1}{4},\frac{1}{4 \sqrt{3}})$. Hence, we obtain
\begin{align} \label{eq451}
&\int\min_{a\in \ga}\|x-a\|^2 dP\\
&\geq  \mathop{\int}\limits_{\tri_{12}\uu \tri_{21}} \min_{a\in \ga}\|x-a\|^2 dP+\mathop{\int}\limits_{\tri_{13}\uu \tri_{23}} \min_{a\in \ga}\|x-a\|^2 dP+\mathop{\int}\limits_{\tri_{11}\uu \tri_{22}} \min_{a\in \ga}\|x-a\|^2 dP\notag\\
&\geq 2\mathop{\int}\limits_{\tri_{12}} \|x-(\frac{1}{2},\frac{1}{18 \sqrt{3}})\|^2 dP+2 \mathop{\int}\limits_{\tri_{13}}\|x-(\frac{2}{9}, \frac{2 \sqrt{3}}{9})\|^2 dP+2 \mathop{\int}\limits_{\tri_{11}} \|x-(\frac{1}{4},\frac{1}{4 \sqrt{3}})\|^2 dP\notag\\
&=\frac{25}{2187}+\frac{5}{729}+\frac{17}{1458}=\frac{131}{4374}=0.0299497>V_3,\notag
\end{align}
which is a contradiction. Thus, we arrive at a contradiction under the assumption that $\ga$ does not contain any point from $\tri_1\uu \tri_2\uu\tri_3$.
Hence, $\ga$ contains at least one point from $\mathop{\uu}\limits_{i=1}^3\tri_i$.
Due to symmetry without any loss of generality we can assume that $\ga$ contains at least one point from $\tri_3$ and does not contain any point from $\tri_1\uu \tri_2$. Then, notice that Voronoi region of any point of $\ga$ which are below the line $x_2=\frac {2\sqrt{3}}{9}$ does not contain any point from $\tri_3$; if it does then the quantization error can be strictly reduced by relocating the points, and it will contradict the fact that $\ga$ is an optimal set.  Hence, if $\ga$ contains two or more points below the line $x_2=\frac {2\sqrt{3}}{9}$, quantization error can be strictly reduced by moving points to $\tri_1$ and to $\tri_2$. So, we assume that $\ga$ contains only one point below the line  $x_2=\frac {2\sqrt{3}}{9}$.
Then as shown in \eqref{eq451}, we have
the distortion error as
\begin{align*}
&\int\min_{a\in \ga}\|x-a\|^2 dP\geq  \frac{131}{4374}=0.0299497>V_3,
\end{align*}
which is a contradiction. Thus, we conclude that $\ga$ does not contain any point from $\tri$ below the line $x_2=\frac {2\sqrt{3}}{9}$. But, then,
\begin{align*}
&\int\min_{a\in \ga}\|x-a\|^2 dP\geq  \mathop{\int}\limits_{\tri_{1}\uu \tri_{2}} \min_{a\in \ga}\|x-a\|^2 dP\\
& \geq 2\Big(\mathop{\int}\limits_{\tri_{11}\uu\tri_{13}}\|x-(\frac{2}{9}, \frac{2 \sqrt{3}}{9})\|^2 dP+\mathop{\int}\limits_{\tri_{12}}\|x-(\frac{5}{18}, \frac{2 \sqrt{3}}{9})\|^2 dP\Big)=\frac{101}{1458}=0.069273,
\end{align*}
which is larger than $V_3$, and so another contradiction arises. All these contradictions arise due to our assumption that $\ga$ contains at least one point from $\tri_3$, and does not contain any point from $\tri_1\uu\tri_2$. We now assume that $\ga$ contains points from any two of the basic triangles $\tri_1$, $\tri_2$ and $\tri_3$. Due to symmetry, without any loss of generality, we can now assume that $\ga$ contains points from $\tri_1$ and $\tri_2$, but does not contain any point from $\tri_3$. In this situation, suppose that $\ga$ does not contain any point above the line  $x_2=\frac {2\sqrt{3}}{9}$.
Then, for any $(x_1, x_2) \in \tri_{31}\uu\tri_{32}$, we have $\min_{(a, b)\in \ga}\|(x_1, x_2)-(a, b)\|^2\geq (\frac{\sqrt{3}}{3}-\frac{2\sqrt{3}}{9})^2=\frac1{27};$ and for any $(x_1, x_2) \in \tri_{33}$, we have $\min_{(a, b)\in \ga}\|(x_1, x_2)-(a, b)\|^2\geq (\frac{4\sqrt{3}}{9}-\frac{2\sqrt{3}}{9})^2= \frac4{27}$. Thus, the distortion error is obtained as
\begin{align*}
&\int\min_{a\in \ga}\|x-a\|^2 dP>\mathop{\int}\limits_{\tri_{31}\uu\tri_{32}} \min_{a\in \ga}\|x-a\|^2 dP+\mathop{\int}\limits_{\tri_{33}}\min_{a\in \ga}\|x-a\|^2 dP\\
&\geq \frac 29 \cdot\frac{1}{27}+\frac 19 \cdot\frac{4}{27} =\frac{2}{81}=0.0246914>V_3
\end{align*}
which is a contradiction. So, we can assume that $\ga$ contains at least one point above the line  $x_2=\frac {2\sqrt{3}}{9}$. Moreover, $\ga$ contains points from both $\tri_1$ and $\tri_2$. Now, if $\ga$ contains only one point above the line $x_2=\frac {2\sqrt{3}}{9}$, then the quantization error can be strictly reduced by moving the point to $\tri_3$. If $\ga$ contains two or more points above the line $x_2=\frac {2\sqrt{3}}{9}$, then the quantization error can be strictly reduced by moving at least one point which are above the line $x_2=\frac {2\sqrt{3}}{9}$ to $\tri_3$. This contradicts the fact that $\ga$ is an optimal set of $n$-means with $n\geq 3$. Hence, $\ga$ contains points from $\tri_i$ for all $1\leq i\leq 3$, i.e., $\ga\ii \tri_i\neq \es$ for all $1\leq i\leq 3$.
\end{proof}

\begin{lemma} \label{lemma31}
Let $\ga$ be an optimal set of $n$-means with $n\geq 3$ and let $n_k=\te{card}(\ga\ii\tri_k)$, $1\leq k\leq 3$.  Then, $\ga\sci \mathop{\uu}\limits_{i=1}^3 \tri_i$, and $|n_i-n_j|=0, \te{ or } 1$ for $1\leq i\neq j\leq 3.$
\end{lemma}

\begin{proof} We will consider the following three cases:

\tit{Case 1: $n=3 k$ for some positive integer $k\geq 1$. }

In this case, due to symmetry we can assume that $\ga$ contains $k$ points from each of $\tri_i$, otherwise, quantization error can be strictly reduced by redistributing the points in $\ga$ equally among $\tri_i$ for $1\leq i\leq 3$. So, $\ga$ does not contain any point from $\tri\setminus \mathop{\uu}\limits_{i=1}^3 \tri_i$ and $|n_i-n_j|=0$ for $1\leq i\neq j\leq 3$.

\tit{Case 2: $n=3 k+1$ for some positive integer $k\geq 1$. }

In this case, due to symmetry, we can assume that $\ga$ contains $k$ points from each of $\tri_i$, and the remaining one point is $(a, b)$. If possible, let $(a, b) \not \in \mathop{\uu}\limits_{i=1}^3 \tri_i$. Due to symmetry we assume that $(a, b)$ lies on the line $x_1=\frac 12$. Then, if $(a, b)$ lies on or above the line $x_2=\frac {\sqrt{3}}{6}$, then $M((a, b)|\ga)$ does not contain any point from $\tri_1\uu\tri_2$. So, quantization error can be strictly reduced by moving the point $(a, b)$ to $\tri_3$, which is a contradiction. We now assume that $(a, b)$ is on the line $x_1=\frac 12$, but below the line $x_2=\frac {\sqrt{3}}{6}$. Notice that if the point $(a, b)$ is below the line $x_2=\frac {\sqrt{3}}{6}$, then $M((a, b)|\ga)$ does not contain any point from $\tri_3$. Let us first assume that $k=1$, i.e., $\ga$ contains only one point from each of $\tri_1$, $\tri_2$ and $\tri_3$. Let $(a_i, b_i)$ be the points that $\ga$ contains from $\tri_i$ for $1\leq i\leq 3$. For any position of $(a, b)$ on the line $x_1=\frac 12$, always $\tri_{11}\sci M((a_1, b_1)|\ga)$. If $M((a_1, b_1)|\ga)$ does not contain any point from $\tri_{13}\uu\tri_{12}$, then we have $(a_1, b_1)=a(11)=(\frac{1}{18},\frac{1}{18 \sqrt{3}})$. But, then $M((a, b)|\ga)$ does not contain any point from $\tri_{11}\uu \tri_{13}\uu\tri_{121}$, and so $M((a_1, b_1)|\ga)$ must contain $\tri_{11}\uu\tri_{13}\uu\tri_{121}$. If $M((a_1, b_1)|\ga)$ does not contain any point from $\tri_{122}\uu\tri_{123}$, then, $(a_1, b_1)=a(11, 12, 121)=(\frac{7}{54},\frac{73}{378 \sqrt{3}})$.  But, then if we draw the boundary of the Voronoi regions of $(a_1, b_1)$ and $(a, b)$, we see that $M((a, b)|\ga)$ does not contain any point from $\tri_{11}\uu\tri_{13}\uu\tri_{121}\uu\tri_{123}$ and it covers largest area from $\tri_1$ if $(a, b)=(\frac 12, 0)$. Thus, we can take
\[(a_1, b_1)=a(11,13,121,123)=(\frac{4}{27},\frac{5}{27 \sqrt{3}}) \te{ and } (a, b)=(\frac 12, 0).\]
Write $A:=\tri_{11}\uu\tri_{13}\uu\tri_{121}\uu\tri_{123}$. If $A\sci M((a_1, b_1)|\ga)$ and $\tri_{122}\sci M((a, b)|\ga)$, then the distortion error is obtained as
\begin{align*}
&\int\min_{c\in\ga}\|x-c\|^2 dP=2\Big( \mathop{\int}\limits_{A}\|x-(a_1, b_1)\|^2+ dP+\mathop{\int}\limits_{\tri_{122}}\|x-(a, b)\|^2 dP\Big) +\mathop{\int}\limits_{\tri_{3}}\|x-(a_3, b_3)\|^2 dP\\
&=2\Big( \mathop{\int}\limits_{A}\|x-(\frac{4}{27},\frac{5}{27 \sqrt{3}})\|^2+ dP+\mathop{\int}\limits_{\tri_{122}}\|x-(\frac 12, 0)\|^2 dP\Big) +\mathop{\int}\limits_{\tri_{3}}\|x-a(3)\|^2 dP\\
&=\frac{1100}{59049}=0.0186286,
\end{align*}
which is larger than $0.015775$, where $\frac{23}{1458}=0.015775$ is the distortion error due to the four-point set $\gb$ given by $\gb:=\set{a(13), a(11, 12), a(2), a(3)}.$  But, this contradicts the optimality of $\ga$. Notice that in this calculation we assumed $\tri_{11}\uu\tri_{13}\uu\tri_{121}\uu\tri_{123}\sci M((a_1, b_1)|\ga)$ and $\tri_{122}\sci M((a, b)|\ga)$. If not, then $M((a_1, b_1)|\ga)$ will contain points from $\tri_{122}$, and then the boundary of the Voronoi regions of the points $(a_1, b_1)$ and $(a, b)$ will move further right from the current position, and proceeding similarly we can show that a contradiction arises. Similarly, we can show that if $k\geq 2$, contradiction arises. Thus, the point $(a,b)$ must belong to either $\tri_1$, $\tri_2$, or $\tri_3$, i.e., $\ga$ must contain $(k+1)$ points from one of $\tri_i$ for $1\leq i\leq 3$, and $k$ points from each of the remaining two triangles.

\tit{Case 3: $n=3 k+2$ for some positive integer $k\geq 1$. }

In this case, due to symmetry, we can assume that $\ga$ contains $k$ points from each of $\tri_i$, and the other two points are symmetrically distributed over the triangle $\tri$ with respect to one of the medians, say the median passing through the vertex $(\frac 12, \frac{\sqrt{3}}{2})$. Then, due to symmetry $\ga$ must contain $(k+1)$ points from $\tri_1$ and $(k+1)$ points from $\tri_2$, otherwise quantization error can be strictly reduced by moving one point to $\tri_1$ and one point to $\tri_2$.

Hence, in each case, we see that if $\ga$ is an optimal set of $n$-means with $n\geq 3$, then $\ga\sci \mathop{\uu}\limits_{i=1}^3 \tri_i$, and $|n_i-n_j|=0, \te{ or } 1$ for $1\leq i\neq j\leq 3.$
Thus, the proof of the lemma is complete.
\end{proof}

As an immediate consequence of Lemma 3.4 we obtain the statement below.

\begin{cor} \label{cor21}  The set $\set{a(1), a(2), a(3)}$ is a unique optimal set of three-means for the measure $P$ with quantization error $V_3=\frac 1{54}=0.0185185.$
\end{cor}

The following lemma plays an important role in the sequel.

\begin{lemma} \label{lemma22}
Let $n\geq 3$ and let $\ga$ be an optimal set of $n$-means. For $1\leq i\leq 3$, set $\gb_i:=\ga\ii \tri_i$ and $n_i:=\te{card}(\gb_i)$. Then, $S_i^{-1}(\gb_i)$ is an optimal set of $n_i$-means, and
$V_n=\mathop{\sum}\limits_{i=1}^3 \frac 1{27}V_{n_i}.$
\end{lemma}
\begin{proof} For $n\geq 3$, by Lemma~\ref{lemma30} and Lemma~\ref{lemma31}, we have $\ga=\mathop{\uu}\limits_{i=1}^3\gb_i$, $n=n_1+n_2+n_3$, and so
$
V_n=\mathop{\sum}\limits_{i=1}^3\mathop{\int}\limits_{\tri_i}\mathop{\min}\limits_{a\in \gb_i}\|x-a\|^2 dP$.
If $S_1^{-1}(\gb_1)$ is not an optimal set of $n_1$-means for $P$, then there exists a set $\gg_1\sci \D R^2$ with $\te{card}(\gg_1)=n_1$ such that
$\int \min_{a\in \gg_1} \|x-a\|^2 dP<\int \min_{a\in S_1^{-1}(\gb_1)} \|x-a\|^2 dP$. But then, $\gd:=S_1(\gg_1)\uu \gb_2\uu \gb_3$ is a set of cardinality $n$, and since
\begin{align*}
&\mathop{\int}\limits_{\tri_1}\min_{a\in S_1(\gg_1)}\|x-a\|^2 dP=\mathop{\int}\limits_{\tri_1}\min_{a\in \gg_1}\|x-S_1(a)\|^2 dP=\frac 1 3 \mathop{\int}\limits_{}\min_{a\in \gg_1}\|x-S_1(a)\|^2 d(P\circ S_1^{-1})\\
&=\frac 1 {27} \mathop{\int}\limits_{}\min_{a\in \gg_1}\|x-a\|^2 dP<\frac 1 {27} \mathop{\int}\limits_{}\min_{a\in S_1^{-1}(\gb_1)}\|x-a\|^2 dP=\frac 1 {27} \mathop{\int}\limits_{}\min_{a\in\gb_1}\|x-S_1^{-1}(a)\|^2 dP\\
&=\frac 1 {3} \mathop{\int}\limits_{}\min_{a\in\gb_1}\|x-a\|^2 d(P\circ S_1^{-1})=\mathop{\int}\limits_{\tri_1}\min_{a\in\gb_1}\|x-a\|^2 dP,
\end{align*}
we have
\[\int\min_{a\in \gd}\|x-a\|^2 dP=\mathop{\int}\limits_{\tri_1}\min_{a\in S_1(\gg_1)}\|x-a\|^2 dP+\mathop{\sum}\limits_{i=2}^3 \mathop{\int}\limits_{\tri_i} \min_{a\in \gb_i}\|x-a\|^2 dP<\int\min_{a\in \ga}\|x-a\|^2 dP,\] which contradicts the fact that $\ga$ is an optimal set of $n$-means for $P$. Similarly, it can be proved that $S_2^{-1}(\gb_2)$ and $S_3^{-1}(\gb_3)$ are optimal sets of $n_2$- and $n_3$-means, respectively. Hence, it follows that,
\[V_n=\mathop{\sum}\limits_{i=1}^3\frac 1 3\mathop{\int}\mathop{\min}\limits_{a\in \gb_i}\|x-a\|^2 d(P\circ S_i^{-1})=\mathop{\sum}\limits_{i=1}^3 \frac 1{27}\mathop{\int}\mathop{\min}\limits_{a\in S_i^{-1}(\gb_i)}\|x-a\|^2 dP=\mathop{\sum}\limits_{i=1}^3 \frac 1{27}V_{n_i},\]
which proves the assertion.
\end{proof}

\begin{lemma} \label{lemma9} Let $P=\mathop{\sum}\limits_{\go \in I^k} \frac 1{3^k} P\circ S_\go^{-1}$ for some $k\geq 1$.
Let $\ga$ be an optimal set of $n$-means for the measure $P$. Then, $\set{S_\go(a) : a \in \ga}$ is an optimal set of $n$-means for the image measure $P\circ S_\go^{-1}$. The converse is also true: If $\gb$ is an optimal set of $n$-means for the image measure $P\circ S_\go^{-1}$, then $\set{S_\go^{-1}(a) : a \in \gb}$ is an optimal set of $n$-means for $P$.
\end{lemma}

\begin{proof}
If $\set{S_\go(a) : a \in \ga}$ is not an optimal set of $n$-means for the image measure $P\circ S_\go^{-1}$, then we can find a set $\gg\sci \D R^2$ with card$(\gg)=n$ such that
\begin{align*} \int \min_{a\in \gg} &\|x-a\|^2 d(P\circ S_\go^{-1}) <\int \min_{a\in \ga} \|x-S_\go(a)\|^2 d(P\circ S_\go^{-1}),
\end{align*}
which implies
$\int \min_{a\in \gg} \|S_\go(x)-a\|^2 dP<\int \min_{a\in \ga} \|S_\go(x)-S_\go(a)\|^2 dP,
$
i.e.,
\[\int \min_{a\in  S_\go^{-1}(\gg)} \|x-a\|^2 dP <\int \min_{a\in \ga} \|x-a\|^2 dP.
\]
Notice that $S_\go^{-1}(\gg)$ has cardinality $n$, and so the last inequality contradicts the fact that $\ga$ is an optimal set of $n$-means for $P$. Hence, $\set{S_\go(a) : a \in \ga}$ is an optimal set of $n$-means for the image measure $P\circ S_\go^{-1}$. To prove the converse, let $\gb$ be an optimal set of $n$-means for the image measure $P\circ S_\go^{-1}$.  If $S_\go^{-1}(\gb)$ is not an optimal set of $n$-means for $P$, then there exists a set $\gd\sci \D R^2$ with $\te{card}(\gd)=n$ such that
$\int \min_{a\in \gd} \|x-a\|^2 dP<\int \min_{a\in S_\go^{-1}(\gb)} \|x-a\|^2 dP,
$
which implies
\[\int \min_{a\in \gd} \|S_\go (x)-S_\go(a)\|^2 dP<\int \min_{a\in  S_\go^{-1}(\gb)} \|S_\go(x)-S_\go(a)\|^2 dP\]
i.e.,
\[\int \min_{a\in  S_\go(\gd)} \|x-a\|^2 d(P\circ S_\go^{-1}) <\int \min_{a \in \gb} \|x-a\|^2 d(P\circ S_\go^{-1}).\]
Notice that $S_\go(\gd)$ has cardinality $n$, and so the last inequality contradicts the fact that $\gb$ is an optimal set of $n$-means for $P\circ S_\go^{-1}$. Thus, we deduce that $\set{S_\go^{-1}(a) : a \in \gb}$ is an optimal set of $n$-means for $P$ if $\gb$ is an optimal set of $n$-means for the image measure $P\circ S_\go^{-1}$. 
\end{proof}

\begin{remark} If $\gb$ is an optimal set of $n$-means for the image measure $P\circ S_\go^{-1}$, and $\gg$ is an optimal set of $\ell$-means for the image measure $P\circ S_\gt^{-1}$, then   $S_\go^{-1}(\gb) \uu S_\gt^{-1}(\gg)$ is not necessarily an optimal set of $(n+\ell)$-means for $P$.

\end{remark}
\begin{lemma}\label{lemma10}
The set $\set{a(1), a(2), a(33), a(31, 32)}$ is an optimal set of four-means with quantization error $V_4=\frac 1{27}(2 V_1+V_2)$.
\end{lemma}
\begin{proof}
Let $\ga$ be an optimal set of four-means. Let $\gb_i=\ga\ii \tri_i$ for $1\leq i\leq 3$. By Lemma~\ref{lemma30} and Lemma~\ref{lemma31}, we can assume that $\te{card}(\gb_1)=\te{card}(\gb_2)=1$ and $\te{card}(\gb_3)=2$, and $\ga=\mathop{\uu}\limits_{i=1}^3\gb_i$. By Lemma~\ref{lemma22}, both $S_1^{-1}(\gb_1)$ and $S_2^{-1}(\gb_2)$ are optimal sets of one-mean, and $S_3^{-1}(\gb_3)$ is an optimal set of two-means. Thus, we can take
$S_1^{-1}(\gb_1)=S_2^{-1}(\gb_2)=(\frac 12, \frac{\sqrt{3}}{6})$, and  $S_3^{-1}(\gb_3)=\set{a(3), a(1, 2)}$ yielding $\gb_1=\set{a(1)}$, $\gb_2=\set{a(2)}$, and $\gb_3=\set{a(33), a(31, 32)}$. By Lemma~\ref{lemma22}, we have the quantization error as $V_4=\frac 1{27}(2 V_1+V_2)$, which completes the proof of the lemma.
\end{proof}

\begin{remark} Due to symmetry, there are nine optimal sets of four-means with quantization error $V_4=\frac{23}{1458}$  (see Figure~\ref{Fig}).
\end{remark}

\begin{lemma} \label{lemma11} Let $n=3^{\ell(n)}+1$ for some positive integer $\ell(n)$. Then, $\set{a(\go) : \go \in I^{\ell(n)}\setminus \set{\gt}}\uu S_\gt(\ga_2)$ is an optimal set of $n$-means for any $\gt\in I^{\ell(n)}$.
\end{lemma}
\begin{proof} We will prove this statement by induction. If $n=4$, then it is true by Lemma~\ref{lemma10}.  Assume that it is true if $n=3^k+1$ for some positive integer $k$. Let $\ga$ be an optimal set of $n$-means for $n=3^{k+1}+1$. Let $\gb_i=\ga\ii \tri_i$ for $1\leq i\leq 3$.  By Lemma~\ref{lemma30} and Lemma~\ref{lemma31}, we can assume that $\te{card}(\gb_1)=\te{card}(\gb_2)=3^{k}$ and $\te{card}(\gb_3)=3^k+1$, and $\ga=\mathop{\uu}\limits_{i=1}^3\gb_i$. Then, by Lemma~\ref{lemma22}, both $S_1^{-1}(\gb_1)$ and $S_2^{-1}(\gb_2)$ are optimal sets of $3^k$-means, and $S_3^{-1}(\gb_3)$ is an optimal of $(3^k+1)$-means. Thus, we can write
$\gb_1=\set{a(1\go) : \go \in I^k}$, $\gb_2=\set{a(2\go) : \go \in I^k}$ and $\gb_3=(\set{a(3\go) : \go\in I^k\setminus \set{\gt}})\uu S_{3\gt}(\ga_2)$ for some $\gt\in I^k$. Hence, $\ga=\set{a(\go) : \go \in I^{k+1}\setminus \set{\gt}}\uu S_\gt(\ga_2)$ for some $\gt\in I^{k+1}$ is an optimal set of $n$-means for $n=3^{k+1}+1$. Thus, by the Principle of Mathematical Induction, the proof of the lemma is complete.
\end{proof}

Now we prove the following propositions which provide further information on the optimal sets of $n$-means.

\begin{prop}\label{prop24}
Let $n\in \D N$ be such that $n=3^{\ell(n)}$ for some positive integer $\ell(n)$. Then, the set $\ga_{3^{\ell(n)}}:=\set{ a(\go) : \go \in I^{\ell(n)}}$ is a unique optimal set of $n$-means for $P$ with quantization error $V_n=\frac 16 \frac1{9^{\ell(n)}}$.
\end{prop}

\begin{proof} 
By Corollary~\ref{cor21}, the assertion is true if $\ell(n)=1$. Let us assume that it is true for $n=3^k$ for some positive integer $k$. We now show that it is also true if $n=3^{k+1}$. Let $\gb$ be an optimal set of $3^{k+1}$-means. Set $\gb_i:=\gb \ii \tri_i$ for $1\leq i\leq 3$. Notice that $\te{card}(\gb_i)=3^k$. Then, by Lemma~\ref{lemma30} and Lemma~\ref{lemma31} and Lemma~\ref{lemma22}, $S_i^{-1}(\gb_i)$ is an optimal set of $3^k$-means, and so
$S_i^{-1}(\gb_i)=\set{a(\go) : \go \in I^k}$ which implies $\gb_i=\set{a(i\go) : \go \in I^k}$. Thus, $\gb=\gb_1\uu \gb_2\uu\gb_3=\set{a(\go) : \go \in I^{k+1}}$ is an optimal set of $3^{k+1}$-means. Since $a(\go)$ is the centroid of $\tri_\go$ for each $\go\in I^{k+1}$, the set $\gb$ is unique. Now, by Lemma~\ref{lemma22}, we have the quantization error as
\[V_{3^{k+1}}=\mathop{\sum}\limits_{i=1}^3 \frac 1{27}V_{3^k}=\frac 1 9 \cdot \frac 16\cdot \frac1{9^k}=\frac 16 \frac{1}{9^{k+1}}. \]
Thus, by the Principle of Mathematical Induction, the proof of the proposition is complete.
\end{proof}

\begin{prop} \label{prop25}
Let $3^{\ell(n)}<n\leq  2 \cdot 3^{\ell(n)}$ for some positive integer $\ell(n)$. Choose $J\sci I^{\ell(n)}$ with $\te{card}(J)=n-3^{\ell(n)}$, and then the set
\[\ga_n(J):=\set{a(\go) : \go \in I^{\ell(n)}\setminus J}\uu\mathop{\uu}\limits_{\go \in J} S_\go(\ga_2) \]
is an optimal set of $n$-means for the measure  $P.$
\end{prop}

\begin{proof} 
If
$3<n\leq 2\cdot3$, the proposition can be proved by proceeding as in Lemma~\ref{lemma11}; hence, it is true if $\ell(n)=1.$
Let the proposition be true if $\ell(n)=m$ for some positive integer $m$.
Let $\gb$ be an optimal set of $n$-means where $n=3^{m+1}+k$ and $1\leq k\leq 3^{m+1}$. Let $J\sci I^{m+1}$ be such that $\te{card}(J)=k$ for $1\leq k\leq 3^{m+1}$.
 Set $\gb_i:=\gb \ii \tri_i$ for $1\leq i\leq 3$. Then, $\gb=\mathop{\uu}\limits_{i=1}^3 \gb_i$ and $\te{card}(\gb_i)=3^{m}+k_i$, where $k_i:=\te{card}(\set{\go \in J : a(\go) \in \tri_i\ii \gb_i})$ for $1\leq i\leq 3$. Notice that $0\leq k_i\leq 3^m$ and $k=k_1+k_2+k_3$. By Lemma~\ref{lemma22}, $S_i^{-1}(\gb_i)$ is an optimal sets of $(3^{m}+k_i)$-means, and so we can write
\[S_i^{-1}(\gb_i)=\set{a(\go) : \go \in I^{m}\setminus J_i}\uu \mathop{\uu}\limits_{\go \in J_i} S_\go(\ga_2) \]
where $J_i \sci I^{m}$ with $\te{card}(J_i)=k_i$. Notice that if $\te{card}(J_i)=0$ then the set $\mathop{\uu}\limits_{\go\in J_i} S_\go(\ga_2)$ is an empty set. Thus, we have
$ \gb_i=\set{a(i\go) : \go \in I^{m}\setminus J_i}\uu \mathop{\uu}\limits_{\go \in J_i} S_{i\go}(\ga_2). $
Hence, $\gb=\gb_1\uu\gb_2\uu\gb_3=\set{a(\go) : \go \in I^{m+1}\setminus J}\uu \mathop{\uu}\limits_{\go \in J} S_\go(\ga_2) $
is an optimal set of $n$-means for $n=3^{m+1}+k$. Therefore, by the Principle of Mathematical Induction, the proposition is true.
\end{proof}

\begin{prop}\label{prop26}  Let $n\in \D N$ be such that $2 \cdot 3^{\ell(n)}<n< 3^{\ell(n)+1}$. Choose $J \sci I^{\ell(n)}$ with $\te{card}(J)=n-2 \cdot 3^{\ell(n)}$, and then the set
\[\ga_n(J):=\mathop{\uu}\limits_{\go\in J} S_\go(\ga_3)\uu \mathop{\uu}\limits_{\go \in I^{\ell(n)}\setminus J} S_\go(\ga_2)\]
is an optimal set of $n$-means for the measure  $P$.
\end{prop}

\begin{proof} Let $n=2 \cdot 3^{\ell(n)}+k$ where $1\leq k< 3^{\ell(n)}$. Let $\gb$ be an optimal set of $n$-means. Write $\gb_i:=\gb\ii \tri_i$ for $1\leq i\leq 3$. First take $\ell(n)=1$, then if $k=1$, by Lemma~\ref{lemma30},  Lemma~\ref{lemma31}, and Lemma~\ref{lemma22}, we can assume that both $S_1^{-1}(\gb_1)$ and $S_2^{-1}(\gb_2)$ are optimal sets of two-means, and $S_3^{-1}(\gb_3)$ is an optimal set of three-means, which yields $\gb=\gb_1\uu \gb_2\uu\gb_3=S_1(\ga_2)\uu S_2(\ga_2)\uu S_3(\ga_3)$, i.e., $\ga_7(\set{3})=S_1(\ga_2)\uu S_2(\ga_2)\uu S_3(\ga_3)$. Thus, the proposition is true if $\ell(n)=1$ and $k=1$. Similarly, we can prove that the proposition is true if $\ell(n)=1$ and $1\leq k<3^{\ell(n)}$. Let us now assume that the proposition is true if $\ell(n)=m$ for some positive integer $m$, where $1\leq k<3^{m}$. Now proceeding as in the proof of Proposition~\ref{prop25}, it can be shown that the proposition is also true for $\ell(n)=m+1$. Therefore, by the Principle of Mathematical Induction, the proposition follows.
\end{proof}

The following theorem which gives all the optimal sets of $n$-means and their numbers, and the corresponding quantization error for all $n\geq 3$.

\begin{theorem}\label{Th1}
For $n\in \D N$ with $n\geq 3$, let $\ell(n)$ be the unique natural number with $3^{\ell(n)}\leq n<3^{\ell(n)+1}$, and $\ga_n$ be an optimal set of $n$-means. If $n=3^{\ell(n)}$, then the set
$\ga_{3^{\ell(n)}}:=\set{ a(\go) : \go \in I^{\ell(n)}}$ is a unique optimal set of $n$-means for $P$.
If $3^{\ell(n)}<n\leq  2\cdot 3^{\ell(n)}$, then the set
$\ga_n(J)=\set{a(\go) : \go \in I^{\ell(n)}\setminus J}\uu\mathop{\uu}\limits_{\go \in J} S_\go(\ga_2)$, where $J \sci I^{\ell(n)}$ with $\te{card}(J)=n-3^{\ell(n)}$,
is an optimal set of $n$-means, and the number of such sets is $^{3^{\ell(n)}}C_{n-3^{\ell(n)}} 3^{n-3^{\ell(n)}}$. On the other hand,
if $2\cdot 3^{\ell(n)}<n< 3^{\ell(n)+1}$, then the set
$\ga_n(J)=\mathop{\uu}\limits_{\go\in J} S_\go(\ga_3)\uu \mathop{\uu}\limits_{\go \in I^{\ell(n)}\setminus J} S_\go(\ga_2)$, where $J \sci I^{\ell(n)}$ with $\te{card}(J)=n-2\cdot 3^{\ell(n)}$,
is an optimal set of $n$-means, and the number of such sets is $^{3^{\ell(n)}}C_{n-2\cdot 3^{\ell(n)}} 3^{ 3^{\ell(n)+1}-n}$. The quantization error is given by
$V_n=
\frac 1 2 \cdot \frac 1{27^{\ell(n)+1}} (13\cdot 3^{\ell(n)}-4n).$

\end{theorem}

\begin{proof} Let us first assume that $n=3^{\ell(n)}$. Then, by Proposition~\ref{prop24}, $\ga_{3^{\ell(n)}}:=\set{ a(\go) : \go \in I^{\ell(n)}}$ is a unique optimal set of $n$-means for $P$ with quantization error
\begin{align*}V_n&= \sum_{\go  \in I^{\ell(n)}}\frac 1{3^{\ell(n)}} \int\|x-a(\go)\|^2  d(P\circ S_\go^{-1})=\sum_{\go  \in I^{\ell(n)}}\frac 1{3^{\ell(n)}}\int\|S_\go(x)-a(\go)\|^2  dP \\
& =\sum_{\go  \in I^{\ell(n)}}\frac 1{3^{\ell(n)}}\frac 1{9^{\ell(n)}} V=\frac 1 6 \frac1{9^{\ell(n)}}=\frac 1 2 \cdot \frac 1{27^{\ell(n)+1}} (13\cdot 3^{\ell(n)}-4n).
\end{align*}
Let us now assume that $3^{\ell(n)}<n\leq  2 \cdot 3^{\ell(n)}$. Then, by Proposition~\ref{prop25},
$\ga_n(J):=\set{a(\go) : \go \in I^{\ell(n)}\setminus J}\uu\mathop{\uu}\limits_{\go \in J} S_\go(\ga_2)$, where $J \sci I^{\ell(n)}$ with $\te{card}(J)=n-3^{\ell(n)}$, is an optimal set of $n$-means. Since the set $J$ from $I^{\ell(n)}$ can be chosen in $^{3^{\ell(n)}}C_{n-3^{\ell(n)}}$ ways and for each $\go\in J$ the set $S_\go(\ga_2)$ can be chosen in three different ways, the number of optimal sets of $n$-means in this case is given by $^{3^{\ell(n)}}C_{n-3^{\ell(n)}} 3^{n-3^{\ell(n)}}$.  The quantization error is
\begin{align*}
& V_n=\int \min_{a\in \ga_n(J)}\|x-a\|^2 dP=\sum_{\go \in I^{\ell(n)}\setminus J}\mathop{\int}\limits_{\tri_\go}  \min_{a\in \ga_n(J)}\|x-a\|^2 dP+\sum_{\go \in J}\mathop{\int}\limits_{\tri_\go}  \min_{a\in \ga_n(J)}\|x-a\|^2 dP\\
&=\sum_{\go \in I^{\ell(n)}\setminus J}\mathop{\int}\limits_{\tri_\go}  \|x-a(\go)\|^2 dP+\sum_{\go \in J}\mathop{\int}\limits_{\tri_\go}  \min_{a\in S_\go(\ga_2)}\|x-a\|^2 dP\\
&=\sum_{\go \in I^{\ell(n)}\setminus J}\frac 1{3^{\ell(n)}} \mathop{\int} \|x-a(\go)\|^2 dP\circ S_\go^{-1}+\sum_{\go \in J}\frac 1{3^{\ell(n)}} \mathop{\int}\min_{a\in S_\go(\ga_2)}\|x-a\|^2 dP\circ S_\go^{-1}\\
&=\sum_{\go \in I^{\ell(n)}\setminus J}\frac 1{3^{\ell(n)}}\frac 1{9^{\ell(n)}} V_1+\sum_{\go \in J}\frac 1{3^{\ell(n)}} \frac 1{9^{\ell(n)}} V_2=\sum_{\go \in I^{\ell(n)}\setminus J}\frac 1{3^{\ell(n)}}\frac 1{9^{\ell(n)}} V+\sum_{\go \in J}\frac 1{3^{\ell(n)}} \frac 1{9^{\ell(n)}} \frac 5 9 V\\
&=\frac 1 6 \cdot \frac 1{27^{\ell(n)}}\Big(\te{card}(I^{\ell(n)}\setminus J)+\frac 59 \te{card}(J)\Big) =\frac 1 6 \cdot \frac 1{27^{\ell(n)}}\Big(2\cdot 3^{\ell(n)}-n +\frac 59 (n-3^{\ell(n)})\Big)\\
&=\frac 1 2 \cdot \frac 1{27^{\ell(n)+1}} (13\cdot 3^{\ell(n)}-4n).
\end{align*}
Let us now assume that $2 \cdot 3^{\ell(n)}<n< 3^{\ell(n)+1}$. Then, by Proposition~\ref{prop26},
$\ga_n(J)=\mathop{\uu}\limits_{\go\in J} S_\go(\ga_3)\uu \mathop{\uu}\limits_{\go \in I^{\ell(n)}\setminus J} S_\go(\ga_2)$, where $J \sci I^{\ell(n)}$ with $\te{card}(J)=n-2\cdot 3^{\ell(n)}$, is an optimal set of $n$-means.
Since the set $J$ from $I^{\ell(n)}$ can be chosen in $^{3^{\ell(n)}}C_{n-2\cdot 3^{\ell(n)}}$ ways and for each $\go\in J$ the set $S_\go(\ga_2)$ can be chosen in three different ways, the number of optimal sets of $n$-means is $^{3^{\ell(n)}}C_{n-2\cdot 3^{\ell(n)}} 3^{ 3^{\ell(n)+1}-n}$, where  $\te{card}(I^{\ell(n)}\setminus J)=3^{\ell(n)}-(n-2\cdot 3^{\ell(n)})=3^{\ell(n)+1}-n$, and the quantization error is
\begin{align*}
& V_n=\sum_{\go \in J}\mathop{\int}\limits_{\tri_\go}  \min_{a\in S_\go(\ga_3)}\|x-a\|^2 dP+\sum_{\go \in I^{\ell(n)}\setminus J}\mathop{\int}\limits_{\tri_\go}  \min_{a\in S_\go(\ga_2)}\|x-a\|^2 dP\\
&=\sum_{\go \in J}\frac 1{3^{\ell(n)}} \mathop{\int} \min_{a\in S_\go(\ga_3)}\|x-a\|^2 dP\circ S_\go^{-1}+\sum_{\go \in I^{\ell(n)}\setminus J}\frac 1{3^{\ell(n)}}\mathop{\int} \min_{a\in S_\go(\ga_2)}\|x-a\|^2 dP\circ S_\go^{-1}\\
&=\sum_{\go \in J}\frac 1{3^{\ell(n)}}\frac 1{9^{\ell(n)}} V_3+\sum_{\go \in I^{\ell(n)}\setminus J}\frac 1{3^{\ell(n)}} \frac 1{9^{\ell(n)}} V_2=\sum_{\go \in J}\frac 1{3^{\ell(n)}}\frac 1{9^{\ell(n)}} \frac 1 9V+\sum_{\go \in I^{\ell(n)}\setminus J}\frac 1{3^{\ell(n)}} \frac 1{9^{\ell(n)}} \frac5 9 V\\
&=\frac 1 2 \cdot \frac 1{27^{\ell(n)+1}}\Big(\te{card}(J)+ 5\,\te{card}(I^{\ell(n)}\setminus J)\Big) =\frac 1 2 \cdot \frac 1{27^{\ell(n)+1}}\Big(n-2 \cdot 3^{\ell(n)}+5(3^{\ell(n)+1}-n)\Big)\\
&=\frac 1 2 \cdot \frac 1{27^{\ell(n)+1}} (13\cdot 3^{\ell(n)}-4n).
\end{align*}
Hence, the proof of the theorem is complete.
\end{proof}

Below, following the results obtained above, we would like to demonstrate how to obtain optimal set of $n$-means by two examples.

\begin{example}
Let $n=11=3^2+2$. Take $J=\set{11, 12}$, where $J\sci I^2$ with $\te{card}(J)=2$. Take $\ga_2=\set{a(1, 2), a(3)}$. Then, by Theorem~\ref{Th1},
\begin{align*}
\ga_{11}(J)&=\set{a(\go) : \go \in I^2\setminus J}\uu \mathop{\uu}\limits_{\go\in J} S_\go(\ga_2)\\
&=\set{a(1,3), a(2,1), a(2,2), a(2,3), a(3,1), a(3,2), a(3, 3)}\\
& \qquad \qquad \uu \set{S_{11}(a(1, 2)), S_{11}(a(3)), S_{12}(a(1,2)), S_{12}(a(3))}\\
&=\set{a(1,3), a(2,1), a(2,2), a(2,3), a(3,1), a(3,2), a(3, 3), \\
& \qquad \qquad a(111, 112), a(113), a(121, 122), a(123)}.
\end{align*}
Using equation~\eqref{eq1}, we obtain the distortion error as
\begin{align*} &\int \min_{a\in \ga_{11}(J)} \|x-a\|^2 dP\\
& =7 \mathop{\int}\limits_{\tri_{13}}(x-a(13))^2 dP+2 \mathop{\int}\limits_{\tri_{113}}(x-a(113))^2 dP+2 \mathop{\int}\limits_{\tri_{111}\uu \tri{112}}(x-a(111, 112))^2 dP=\frac{73}{39366}.
\end{align*}
Now substituting $\ell(n)=2$ and $n=11$ in the formula given by Theorem~\ref{Th1}, we also obtain that
\[V_{11}=\frac 12\cdot \frac{1}{27^3} (13 \cdot 3^2-4\cdot 11)=\frac{73}{39366}.\]
\end{example}

\medskip
\begin{example}
Let $n=19=2\cdot 3^2+1$. Take $J=\set{11}$, where $J\sci I^2$ with $\te{card}(J)=1$. Take $\ga_2=\set{a(1, 2), a(3)}$. Notice that  $\ga_3=\set{a(1), a(2), a(3)}$ which is unique. Then, by Theorem~\ref{Th1}, we have
\begin{align*}
&\ga_{19}(J)=\mathop{\uu}\limits_{\go \in I^2\setminus J}S_\go(\ga_2)\uu \mathop{\uu}\limits_{\go\in J} S_\go(\ga_3)\\
&=\set{a(121, 122), a(123), a(131, 132), a(133), a(211, 212), a(213), a(221, 222), a(223), a(231, 232), \\
& a(233), a(311, 312), a(313), a(321, 322), a(323),
a(331, 332), a(333), a(111), a(112), a(113)}.
\end{align*}
Now substituting $\ell(n)=2$ and $n=19$ in the formula given by Theorem~\ref{Th1}, we obtain
\[V_{19}=\frac 12\cdot \frac{1}{27^3} (13 \cdot 3^2-4\cdot 19)=\frac{41}{39366},\]
which can also be obtained by using equation~\eqref{eq1}.
\end{example}

\bigskip

\section{Quantization dimension and quantization coefficient}

Since the stretched Sierpi\`nski triangle under investigation satisfies the strong separation condition, with each $S_i$ having contracting factor of $\frac{1}{3}, $ its Hausdorff dimension is equal to the similarity dimension.  Hence, from the equation $3 (\frac 1 3)^\gb=1, $ we have $\dim_{\te{H}}(S)=\beta =1. $  By Theorem~14.17  \cite{GL2}, the quantization dimension $D(P) $ exists and is equal to $ \beta=1.$ Moreover, using the formula given by Theorem~A \cite{MR}, we see that the Hausdorff dimension and the packing dimension of the measure $P$ are obtained as one. Thus, for the probability measure $P$ with support the stretched Sierpi\`nski triangle, the Hausdorff dimension, the packing dimension, and the quantization dimension coincide.
In the sequel, we show that the $\gb$-dimensional quantization coefficient for $P$ does not exist.

First, observe that if the function $f : [1, 2]\to \D R$ is defined by $f(x)=\frac 1{54} x^2(13-4x)$, then it is strictly increasing on the interval $[1, 2]$, and $f([1, 2])=[\frac {1}{6}, \frac {10}{27}]$.



\begin{theorem} \label{Th4}
 $\gb$-dimensional quantization coefficient for $\gb=1$ does not exist.
\end{theorem}
\begin{proof} We need to show that $\mathop{\lim}\limits_{n\to \infty} n^2 V_n$ does not exist. Let $(n_k)_{k\in \D N}$ be a subsequence of the set of natural numbers such that $3^{\ell(n_k)}\leq n_k< 3^{\ell(n_k)+1}$. To prove the theorem it is enough to show that the set of accumulation points of the subsequence $(n_k^2 V_{n_k})_{k\geq 1}$ equals $[\frac {1}{6}, \frac {10}{27}]$. Let  $y \in [\frac {1}{6}, \frac {10}{27}]$. We now show that $y$ is a subsequential limit of the sequence $(n_k^2 V_{n_k})_{k\geq 1}$. Since $y\in [\frac {1}{6}, \frac {10}{27}]$, $y=f(x)$ for some $x\in [1, 2]$. Set $n_{k_\ell}=\lfloor x 3^{\ell}\rfloor$, where $\lfloor x 3^{\ell}\rfloor$ denotes the greatest integer less than or equal to $ x 3^{\ell}$. Then, $n_{k_\ell}<n_{k_{\ell+1}}$ and $\ell(n_{k_\ell})=\ell$, and there exists $x_{k_\ell} \in [1, 2]$ such that $n_{k_\ell}=x_{k_\ell} 3^\ell$. Notice that by $\ell(n_{k_\ell})=\ell$ it is meant that $3^\ell\leq n_{k_\ell}<3^{\ell+1}$. Thus, putting the values of $V_{n_{k_\ell}}$ from Theorem~\ref{Th1} we obtain
\begin{align*}
n_{k_\ell}^2V_{n_{k_\ell}}=n_{k_\ell}^2 \frac 1 2 \cdot \frac 1{27^{\ell+1}} (13\cdot 3^{\ell}-4n_{k_\ell}) =x_{k_\ell}^2 9^\ell\frac 1 2 \cdot \frac 1{27^{\ell+1}} (13\cdot 3^{\ell}-4x_{k_\ell} 3^\ell),
\end{align*}
which yields
\begin{align} \label{eq45} n_{k_\ell}^2V_{n_{k_\ell}}=\frac1{54}x_{k_\ell}^2(13-4x_{k_\ell})=f(x_{k_\ell}).\end{align}
Again, $x_{k_\ell} 3^{\ell}\leq x 3^\ell<x_{k_\ell} 3^{\ell}+1$, which implies $x-\frac 1{3^{\ell}}< x_{k_\ell} \leq x$, and so, $\mathop{\lim}\limits_{\ell\to \infty} x_{k_\ell}=x$. Since, $f$ is continuous, we have
\[\mathop{\lim}\limits_{\ell\to \infty}  n_{k_\ell}^2V_{n_{k_\ell}}=f(x)=y,\]
which yields the fact that $y$ is an accumulation point of the subsequence $(n_k^2 V_{n_k})_{k\geq 1}$ whenever $y\in [\frac {1}{6}, \frac {10}{27}]$. To prove the converse, let $y$ be an accumulation point of the subsequence  $(n_k^2 V_{n_k})_{k\geq 1}$. Then, there exists a subsequence  $(n_{k_i}^2 V_{n_{k_i}})_{i\geq 1}$ of  $(n_k^2 V_{n_k})_{k\geq 1}$ such that $\mathop{\lim}\limits_{i\to \infty}n_{k_i}^2 V_{n_{k_i}}=y$. Set $\ell_{k_i}=\ell(n_{k_i})$ and $x_{k_i}=\frac{n_{k_i}}{3^{\ell_{k_i}}}$. Then,  $x_{k_i} \in[1, 2]$, and as shown in \eqref{eq45}, we have
\[n_{k_i}^2 V_{n_{k_i}}=f(x_{k_i}).\]
Let $(x_{k_{i_j}})_{j\geq 1}$ be a convergent subsequence of $(x_{k_i})_{i\geq 1}$, and then we obtain
\[y=\lim_{i\to \infty} n_{k_i}^2 V_{n_{k_i}}=\lim_{j\to \infty}n_{k_{i_j}}^2 V_{n_{k_{i_j}}}=\lim_{j\to \infty}f(x_{k_{i_j}}) \in [\frac {1}{6}, \frac {10}{27}].\]
Thus, we deduce that the set of accumulation points of the subsequence $(n_k^2 V_{n_k})_{k\geq 1}$ is the interval $[\frac {1}{6}, \frac {10}{27}]$; hence, the proof of the theorem is complete.
\end{proof}

\section{Further remarks}
In \cite{BMS} some properties of ``fat" Sierpi\'nski triangles were studied.  These are the attractors of iterated function systems defined by $\{S_i \}_{i=1}^3, $ where
$$S_i (x_1,x_2)=r (x_1,x_2) +(1-r) p_i, \ r \in (\frac12 ,1), $$
and $p_i$ are three non-collinear points in $\D R^2$.
Their focus is on the calculation of the Hausdorff dimension of these fractals and, since such fractals do not satisfy the open set condition (OSC), the calculation of the Hausdorff dimension is highly non-trivial.  They also mention, in passing,  the attractors of the iterated function systems when $r\in (0, 1/2]$ and observe that the resulting fractals satisfy the open set condition, essentially disjoint and have fractal dimension
$\frac{\log 3}{- \log r} . $  Of course, when $0<r<\frac12, $ the fractals are totally disconnected.  The stretched Sierpi\`nski triangle we studied above is actually the case $r=\frac13. $

\begin{remark} Let $0<r_1, r_2, r_3<\frac 12$. Then, a general stretched Sierpi\`nski triangle can be constructed by the contractive mappings $S_1, S_2, S_3$ on $\D R^2$, such that
 $S_1(x_1, x_2)=r_1 (x_1, x_2)$, $S_2(x_1, x_2)=r_2 (x_1, x_2)+ (1-r_2)(1, 0)$, and  $S_3(x_1, x_2)=r_3 (x_1, x_2)+ (1-r_3)(\frac 12, \frac {\sqrt{3}}{2})$ for all $(x_1, x_2) \in \D R^2$; or, by the contractive mappings given by $T_1(x_1, x_2)=r_1 (x_1, x_2)$, $T_2(x_1, x_2)=r_2 (x_1, x_2)+ (1-r_2)(1, 0)$, and  $T_3(x_1, x_2)=r_3 (x_1, x_2)+ (1-r_3)(0, 1)$ for all $(x_1, x_2) \in \D R^2$. A general singular continuous probability measure $P$ on a stretched Sierpi\`nski triangle can be defined by $P=p_1P\circ S_1^{-1}+p_2P\circ S_2^{-1}+p_3P\circ S_3^{-1}$ where $(p_1, p_2, p_3)$ is a probability vector with $p_i>0$ for all $1\leq i\leq 3$. If $r_1=r_2=r_3=r$, then a general stretched Sierpi\`nski triangle reduces to the triangle considered in this paper. For a general probability distribution on a general stretched Sierpi\`nski triangle the optimal sets of $n$-means and the $n$th quantization error are not known yet for all $n\geq 2$.
\end{remark}

\section{Mathematica Code} \label{sec4}
Throughout the paper, for computation, we have used the following Mathematica Code.

\[\te{Code} \left\{\begin{array} {ll}
&r=\frac{1}{3}; \quad S[1][\{\text{x$\_$},\text{y$\_$}\}]:= r \{x,y\}; \quad S[2][\{\text{x$\_$},\text{y$\_$}\}]:=r \{x,y\}+(1-r) \{1,0\};\\
&S[3][\{\text{x$\_$},\text{y$\_$}\}]:=r \{x,y\}+(1-r)  \{\frac{1}{2}, \frac{\sqrt{3}}{2}\};\\
&Ex=\{\frac{1}{2},\frac{\sqrt{3}}{6}\}; \quad   V=\frac{1}{6}; \quad  L[J$\_$]\text{:=}\frac{1}{3^{\text{Length}[J]}}; \quad P[J$\_$]\text{:=}\frac{1}{3^{\text{Length}[J]}};\\
&M[\text{k$\_$}, \text{x$\_$}]:=\text{Module} [\{F, i\}, \ F=x; \  \text{Do} [F=S[i][F], \ \{i,\text{Reverse}[k]\}]; \ F];\\
&Av[x$\_$] := (1/\te{Sum}[P[i], \{i, x\}]) \te{Sum} [P[i]\, M[i, Ex] , \{i, x\}];\\
&Er[x$\_$, y$\_$] :=\te{Sum}[P[i] \Big(L[i]^2 V + \te{SquaredEuclideanDistance}[M[i, Ex], y]\Big), \{i, x\}];
\end{array}
\right.
\]
\subsection*{User Guide} First, one needs to copy and paste the above code in Mathematica Notebook. $S[1]$, $S[2]$, and $S[3]$ represent the three similarity mappings, and $r=\frac 13$ is the similarity ratio of the similarity mappings. $Ex$ and $V$ represent the expectation and the variance as given in Lemma~\ref{lemma333}. Then:

$(i)$ For any two words $\go:=\go_1\go_2\cdots\go_k$ and $\gt:=\gt_1\gt_2\cdots\gt_\ell$, we have
\begin{align*}
E(X : X\in \tri_\go)&=Av[\set{\set{\go_1, \go_2, \cdots, \go_k}}], \\
E(X : X\in \tri_\go\uu \tri_\gt)&=Av[\set{\set{\go_1, \go_2, \cdots, \go_k}, \set{\gt_1, \gt_2, \cdots, \gt_\ell}}].
\end{align*}
The above formula helps us to find the conditional expectations.

$(ii)$ For any $x:=\set{x_1, x_2}\in \D R^2$, and any $\go:=\go_1\go_2\cdots\go_k \in I^\ast$,
\[S_\go(x)=M[\set{\go_1\go_2\cdots\go_k}, \set{x_1, x_2}].\]
The above formula helps us to find the image of any point in $\D R^2$ under any composition mapping $S_\go:=S_{\go_1}\circ S_{\go_2}\circ \cdots \circ S_{\go_k}$.

$(iii)$  For any two words $\go:=\go_1\go_2\cdots\go_k$ and $\gt:=\gt_1\gt_2\cdots\gt_\ell$, and a point $a:=\set{a_1, a_2}\in \D R^2$, we have
\begin{align*}
\int_{\tri_\go}\|x-a\|^2 dP&=Er[\set{\set{\go_1, \go_2, \cdots, \go_k}}, \set{a_1, a_2}], \\
\int_{\tri_\go\uu \tri_\gt}\|x-a\|^2 dP&=Er[\set{\set{\go_1, \go_2, \cdots, \go_k}, \set{\gt_1, \gt_2, \cdots, \gt_\ell}}, \set{a_1, a_2}].
\end{align*}
The above formula helps us to find the distortion errors.
\qed
\subsection*{Acknowledgement} The authors are grateful to the referees for their valuable comments and suggestions.

\end{document}